\documentclass[10pt,reqno]{amsart}
\usepackage{amsmath,amsthm,amssymb,amsfonts,amscd}
\usepackage{mathrsfs}
\usepackage{bbding}
\usepackage{graphicx,latexsym}
\usepackage{hyperref}
\usepackage{geometry}
\usepackage{color}
\usepackage{xcolor}

\numberwithin{equation}{section}

\theoremstyle{plain}
\newtheorem{theorem}{Theorem}[section]
\newtheorem*{theorem*}{Theorem}
\newtheorem{lemma}[theorem]{Lemma}

\newtheorem{proposition}[theorem]{Proposition}

\theoremstyle{definition}

\theoremstyle{remark}
\newtheorem{remark}{Remark}

\renewcommand{\Re}{\operatorname{Re}}
\renewcommand{\Im}{\operatorname{Im}}

\newcommand{\diag}{\operatorname{diag}}

\renewcommand{\mod}{\operatorname{mod}\ }
\newcommand{\bs}{\backslash}
\newcommand{\la}{\langle}
\newcommand{\ra}{\rangle}

\newcommand{\cA}{\mathcal{A}}
\newcommand{\cB}{\mathcal{B}}
\newcommand{\cC}{\mathcal{C}}
\newcommand{\cD}{\mathcal{D}}
\newcommand{\cE}{\mathcal{E}}

\newcommand{\cL}{\mathcal{L}}
\newcommand{\cM}{\mathcal{M}}

\newcommand{\cR}{\mathcal{R}}
\newcommand{\cS}{\mathcal{S}}
\newcommand{\cT}{\mathcal{T}}

\newcommand{\cV}{\mathcal{V}}
\newcommand{\cW}{\mathcal{W}}

\newcommand{\fa}{\mathfrak{a}}

\newcommand*{\bbC}{\ensuremath{\mathbb{C}}}

\newcommand*{\bbR}{\ensuremath{\mathbb{R}}}

\newcommand*{\bbZ}{\ensuremath{\mathbb{Z}}}
\newcommand*{\bbN}{\ensuremath{\mathbb{N}}}

\newcommand*{\bbH}{\ensuremath{\mathbb{H}}}

\newcommand{\D}{\Delta}

\newcommand{\ve}{\varepsilon}

\makeatletter
\def\@tocline#1#2#3#4#5#6#7{\relax
  \ifnum #1>\c@tocdepth 
  \else
    \par \addpenalty\@secpenalty\addvspace{#2}%
    \begingroup \hyphenpenalty\@M
    \@ifempty{#4}{%
      \@tempdima\csname r@tocindent\number#1\endcsname\relax
    }{%
      \@tempdima#4\relax
    }%
    \parindent\z@ \leftskip#3\relax \advance\leftskip\@tempdima\relax
    \rightskip\@pnumwidth plus4em \parfillskip-\@pnumwidth
    #5\leavevmode\hskip-\@tempdima
      \ifcase #1
       \or\or \hskip 1em \or \hskip 2em \else \hskip 3em \fi%
      #6\nobreak\relax
    \hfill\hbox to\@pnumwidth{\@tocpagenum{#7}}\par
    \nobreak
    \endgroup
  \fi}
\makeatother

\begin{document}

\title[Hybrid subconvexity bounds for twisted $L$-functions on $GL(3)$]
        {Hybrid subconvexity bounds for twisted $L$-functions on $GL(3)$}
\author{Bingrong Huang}
\address{School of Mathematics \\ Shandong University \\ Jinan \\Shandong 250100 \\China}
\email{bingronghuangsdu@gmail.com}
\address{Current address: School of Mathematical Sciences \\ Tel Aviv University \\ Tel Aviv 69978 \\ Israel}
\date{\today}
\thanks{Project is partly supported by NSFC grant 11531008 and
IRT\_16R43 from the Ministry of Education, China.}

\begin{abstract}
  Let $q$ be a large prime, and $\chi$ the quadratic character modulo $q$.
  Let $\phi$ be a self-dual Hecke--Maass cusp form for $SL(3,\bbZ)$,
  and $u_j$ a Hecke--Maass cusp form 
  for $\Gamma_0(q)\subseteq SL(2,\bbZ)$
  with spectral parameter $t_j$.
  We prove, for the first time, some hybrid subconvexity bounds for the twisted $L$-functions on $GL(3)$, such as
  \[
    L(1/2,\phi\times u_j\times\chi)\ll_{\phi,\ve} (q(1+|t_j|))^{3/2-\theta+\ve},\quad
    L(1/2+it,\phi\times\chi)\ll_{\phi,\ve} (q(1+|t|))^{3/4-\theta/2+\ve},
  \]
  for any $\ve>0$, where $\theta=1/23$ is admissible.
  The proofs depend on the first moment of a family of $L$-functions in short intervals. In order to bound this moment, we first use the approximate functional equations, the Kuznetsov formula, and the Voronoi formula to transform it to a complicated summation; and then we  apply different methods to estimate it, which give us strong bounds in different aspects. We also use the stationary phase method and the large sieve inequalities.
\end{abstract}

\keywords{$L$-functions, subconvexity, $GL(3)$, twisted, quadratic character}
\subjclass{11F66, 11F67, 11M41}
\maketitle
\setcounter{tocdepth}{1}
{\small\tableofcontents}

\section{Introduction} \label{sec: Introduction}

Bounding $L$-functions on their critical lines is one of the central problems
in analytic number theory.
For $GL(1)$ $L$-functions,
subconvexity bounds are due to Weyl~\cite{weyl1921abschatzung} in the $t$-aspect,
and Burgess~\cite{burgess1963character} in the $q$-aspect.
Hybrid bounds for Dirichlet $L$-functions are given by
Heath-Brown~\cite{heath1978hybrid,heath1980hybrid}.
For $GL(2)$ $L$-functions,
in the weight aspect, this was achieved in Peng~\cite{peng2001zeros}.
In the conductor aspect, Conrey--Iwaniec~\cite{conrey2000cubic} used
the cubic moment to give a strong subconvexity bound.
And recently, Young~\cite{young2014weyl} generalized their method to obtain
a Weyl-type hybrid subconvexity bounds for twisted $L$-functions.
In the level aspect, this was first given by
Duke--Friedlander--Iwaniec~\cite{duke1994bounds}.
Subconvexity bounds for Rankin--Selberg $L$-functions on $GL(2)\times GL(2)$ were known
due to Sarnak~\cite{sarnak2001estimates}, Kowalski--Michel--Vanderkam~\cite{kowalski2002rankin},
and Lau--Liu--Ye~\cite{lau2006new}, etc.
Now for $L$-functions on $GL(1)$ and $GL(2)$, this was solved completely,
due to the work of Michel--Venkatesh~\cite{michel2010subconvexity} and many other important
contributions on the way.
For $GL(3)$ $L$-functions,
Li~\cite{li2011bounds} gave the first subconvexity bound in the $t$-aspect for self-dual forms.
Recently, McKee--Sun--Ye~\cite{mckee2015improved} improved Li's results.
Blomer~\cite{blomer2012subconvexity} considered the conductor aspect for twisted $L$-functions on $GL(3)$.
On the other hand, in a series of papers~\cite{munshi2015circle2,munshi2015circle3,munshi2015circle4},
Munshi used the circle method and $GL(3)$ Voronoi formula to give the subconvexity bounds.
So far, there are mainly two methods to solve the subconvexity problem for $GL(3)$ $L$-functions:
the moment method and the circle method. They work in different situations.

In this paper, we consider certain types of twisted $L$-functions
of degree $3$ and $6$ in both $q$ and $t$ aspects.
More precisely, let $q$ be a large prime, and
$\chi$ the primitive quadratic character modulo $q$.
Let $u_j$ be an even Hecke--Maass cusp newform with spectral parameter
$t_j$ of level $q'|q$.
We denote the Hecke eigenvalues by $\lambda_j(n)$.
Let $\phi$ be a self-dual Hecke--Maass form of type $(\nu,\nu)$
for $SL(3,\bbZ)$, with Fourier coefficients $A(m,n)=A(n,m)$, normalized
so that the first Fourier coefficient $A(1,1)=1$.
We define the $L$-function
\begin{equation*}
  L(s,\phi) = \sum_{n=1}^{\infty} \frac{A(1,n)}{n^s},
\end{equation*}
for $\Re(s)>1$. The twisted $L$-functions
\begin{equation*}\label{eqn: L(s,phi.chi)}
  L(s,\phi\times\chi) = \sum_{n=1}^{\infty} \frac{A(1,n)\chi(n)}{n^s}
\end{equation*}
is defined for $\Re(s)>1$, and can be continued to an entire function
with a functional equation of conductor $q^3$.
Similarly, we define the Rankin--Selberg $L$-function
\begin{equation*}\label{eqn: L(s,phi.u.chi)}
  L(s,\phi\times u_j\times\chi) = \sum_{m\geq1}\sum_{n\geq1}\frac{A(m,n)\lambda_j(n)\chi(n)}{(m^2n)^s},
\end{equation*}
for $\Re(s)>1$, and can be continued to an entire function with conductor $q^6$.

Our main result is
\begin{theorem}\label{thm: main}
  With notation as above, we have
  \begin{equation*}
    L(1/2,\phi\times u_j\times\chi) \ll_{\phi,\ve} (q(1+|t_j|))^{3/2-\theta+\ve},
  \end{equation*}
  and
  \begin{equation*}
    L(1/2+it,\phi\times\chi) \ll_{\phi,\ve} (q(1+|t|))^{3/4-\theta/2+\ve},
  \end{equation*}
  for any $\ve>0$, where $\theta=(35-\sqrt{1057})/56$. 
\end{theorem}

In order to prove Theorem \ref{thm: main}, we will use
two different methods to show the following two theorems.
And  with some modifications, we will give the proof of Theorem \ref{thm: main}
in the end of \S\ref{sec: MT}.

\begin{theorem}\label{thm: q}
  With notation as above, we have
  \begin{equation*}
    L(1/2,\phi\times u_j\times\chi) \ll_{\phi,\ve} q^{5/4+\ve}(1+|t_j|)^{3/2+\ve},
  \end{equation*}
  and
  \begin{equation*}
    L(1/2+it,\phi\times\chi) \ll_{\phi,\ve} q^{5/8+\ve}(1+|t|)^{3/4+\ve},
  \end{equation*}
  for any $\ve>0$.
\end{theorem}

and

\begin{theorem}\label{thm: t}
  With notation as above, we have
  \begin{equation*}
    L(1/2,\phi\times u_j\times\chi) \ll_{\phi,\ve} q^{4+\ve}(1+|t_j|)^{4/3+\ve},
  \end{equation*}
  and
  \begin{equation*}
    L(1/2+it,\phi\times\chi) \ll_{\phi,\ve} q^{2+\ve}(1+|t|)^{2/3+\ve},
  \end{equation*}
  for any $\ve>0$.
\end{theorem}


\begin{remark}
  Theorem \ref{thm: main} gives the first hybrid subconvexity
  bound for $GL(3)$ $L$-functions.
  Note that the convexity bound for $L(1/2,\phi\times u_j\times\chi)$
  is $(q(1+|t_j|))^{3/2+\ve}$,
  and for $L(1/2+it,\phi\times\chi)$ is $(q(1+|t|))^{3/4+\ve}$.
  Theorem \ref{thm: q} is crucial, which is a generalization of Blomer's
  results in~\cite{blomer2012subconvexity}, since the bounds there are
  subconvexity in the $q$-aspect and convexity in the $t$-aspect.
  So any bound which is subconvexity in terms of $t$
  and of polynomial growth in terms of $q$ is sufficient to
  get a hybrid subconvexity bound by combining with Theorem \ref{thm: q}.
  Theorem \ref{thm: t} is a generalization of Li's
  results in~\cite{li2011bounds} and McKee--Sun--Ye's
  improvements in~\cite{mckee2015improved}.
\end{remark}

\begin{remark}
  We emphasize that we do not expect our theorems
  to be optimal. For example, the bounds in the $q$-aspect in Theorem \ref{thm: t}
  may be improved if one can use the large sieve inequality for the $d_2'$-sum
  in \eqref{eqn: R_j^d}, and so can Theorem \ref{thm: main}.
\end{remark}

\begin{remark}
  Let $f$ be a weight $2k$ holomorphic modular form for $\Gamma_0(q)$.
  One may prove
  \begin{equation*}
    L(1/2,\phi\times f\times \chi)\ll_{\phi,\ve} (qk)^{3/2-\theta+\ve}.
  \end{equation*}
  The proof of the above result is similar to Theorem \ref{thm: main},
  see Li~\cite[Appendix]{li2011bounds} for example.
  One can also think about the hybrid subconvexity bounds for $GL(3)$ $L$-functions
  in other cases, such as Munshi~\cite{munshi2013bounds1,munshi2013bounds3,munshi2015circle4}.
\end{remark}

We give a brief outline of proofs of our theorems.
Let $T \asymp t_j \asymp t$ be large positive numbers.
In our work, we will assume $q\ll T^B$, for some fixed $B>0$.
Note that Blomer's method showed an upper bound of the form $T^A q^{5/4+\ve}$.
To prove our theorems, the basic idea is similar to
Li~\cite{li2011bounds} and Blomer~\cite{blomer2012subconvexity}.
We consider the average of $L(1/2,\phi\times u_j\times \chi)$
over the spectrum of the Laplacian on $\Gamma_0(q)\bs\bbH$,
see Proposition \ref{prop: q} and \ref{prop: t} below.
And then our results follow from a theorem of Lapid~\cite{lapid2003nonnegativity},
which shows that $L(1/2,\phi\times u_j\times \chi)$ is always a non-negative
real number. (We can drop all but one term to obtain an individual bound;
similarly for $L(1/2+it,\phi\times\chi)$.)
To prove Proposition \ref{prop: q}, which is strong in $q$-aspect,
after applying the approximate functional equations for the Rankin--Selberg
$L$-functions, the $GL(2)$ Kuznetsov formula,
and the $GL(3)$ Voronoi formula, we are led to
bound $\cS_\sigma(q,N;\delta)$, see \eqref{eqn: cS_s}.
To estimate $\cS_\sigma(q,N;\delta)$, we will use the hybrid large
sieve inequality and many results in Conrey--Iwaniec~\cite{conrey2000cubic},
Blomer~\cite{blomer2012subconvexity}, and Young~\cite{young2014weyl}.
This is inspired by Young~\cite{young2014weyl}.
However, this will not give us subconvexity bounds in both
$q$ and $t$ aspects. 
In order to prove results as in Theorem \ref{thm: main},
we still need to handle the case $q$ is much smaller than $t$. 
That is, we will need a result as in Theorem \ref{thm: t},
which is strong in the $t$-aspect and will follow from
Proposition \ref{prop: t}. Now, to prove Proposition \ref{prop: t},
it turns out that Li's method still works.
The key point here is that we can have a second application
of the Voronoi formula. To get a better bound, we will also use
an $n$th-order asymptotic expansion of a weighted stationary phase
integral as McKee--Sun--Ye~\cite{mckee2015improved} did.
Throughout the paper, $e(x)$ means $e^{2\pi ix}$,
negligible means $O(T^{-A})$ for any $A>0$,
and $\ve$ is an arbitrarily small positive number which
may not be the same in each occurrence.


This paper is organized as follows. In \S\ref{sec: preliminaries}, we introduce some notation and present some lemmas we will need later. A setup of our first main result is given in \S\ref{sec: setup q}, which leads us to consider the weight function $\Psi_{\sigma}^{\pm}$ and the character sum $\mathcal T^{\pm,\sigma}$ (see \eqref{eqn: cS_s} below). We then apply some analytic methods to separate variables in the weight function in \S\ref{sec: Psi}, and use the hybrid large sieve inequality to give an inequality involving the character sums in \S\ref{sec: large_sieve}. In \S\ref{sec: thm q}, we give the proof of Theorem \ref{thm: q}.
Then we switch to another method to estimate the first moment. After a simple setup in \S\ref{sec: setup t}, we handle the new weight function in \S\ref{sec: Psi II} and bound the error terms in \S\ref{sec: ET}. Finally, the proofs of the other main theorems are given in \S\ref{sec: MT}.

\section{Preliminaries}\label{sec: preliminaries}

In this section, we introduce notation and recall some standard facts of
automorphic forms on $GL(2)$ and $GL(3)$.

\subsection{Automorphic forms}

We start by reviewing automorphic forms for $\Gamma_0(q)$.
Let $\mathbb{H}$ be the upper half-plane.
Let $\cA(\Gamma_0(q)\bs\bbH)$ denote the space of automorphic functions of weight zero, i.e.,
the functions $f:\bbH\rightarrow\bbC$ which are $\Gamma_0(q)$-periodic.
Let $\cL(\Gamma_0(q)\bs\bbH)$ denote the subspace of square-integrable functions
with respect to the inner product
\begin{equation*}\label{eqn: inner product}
  \la f,g \ra = \int_{\Gamma_0(q)\bs\bbH} f(z)\overline{g(z)}d\mu z,
\end{equation*}
where $d\mu z = y^{-2}dxdy$ is the invariant measure on $\bbH$.
The Laplace operator
\begin{equation*}
  \D = -y^2\left(\frac{\partial^2}{\partial x^2} + \frac{\partial^2}{\partial y^2} \right)
\end{equation*}
acts in the dense subspace of smooth functions in $\cL(\Gamma_0(q)\bs\bbH)$ such that $f$ and
$\D f$ are both bounded; it has a self-adjoint extension which yields the spectral
decomposition
$$
  \cL(\Gamma_0(q)\bs\bbH) = \bbC\oplus\cC(\Gamma_0(q)\bs\bbH)\oplus\cE(\Gamma_0(q)\bs\bbH).
$$
Here $\bbC$ is the space of constant functions,
$\cC(\Gamma_0(q)\bs\bbH)$ is the space of cusp forms and
$\cE(\Gamma_0(q)\bs\bbH)$ is the space of Eisenstein series.

We choose an orthonormal basis $\cB(q)$ of even Hecke--Maass forms of level $q$ as follows:
for each even newform $u_j$ of level $q'|q$ we choose an orthonormal basis $\cV(u_j)$
of the space generated by $\{u_j(dz):d|(q/q')\}$ containing $u_j/\|u_j\|$,
and let $\cB(q)$ be the union of all $\cV(u_j)$ for $u_j$ ranging over the newforms of level dividing $q$.
Let $\cB^*(q)$ be the subset of all newforms in $\cB(q)$.
Each $u_j\in\cB(q)$ with spectral parameter $t_j$ has a Fourier expansion
\[
  u_j(z) = \sum_{n\neq0} \rho_j(n)W_{s_j}(nz),
\]
where $W_s(z)$ is the $GL(2)$ Whittaker function given by
\[
  W_s(z) := 2|y|^{1/2} K_{s-1/2}(2\pi|y|)e(x),
\]
and $K_s(y)$ is the $K$-Bessel function with $s=1/2+it$.
We have the Hecke operators acting on $u_j$ with
\begin{equation*}\label{eqn: HO}
  (T_n u_j)(z) := \frac{1}{\sqrt{n}}\sum_{ad=n}
                    \sum_{b(\mod d)}u_j\left(\frac{az+b}{d}\right)
                = \lambda_j(n)u_j(z),
\end{equation*}
for all $n$ with $(n,q)=1$.
We have
\[
  \rho_j(\pm n) = \rho_j(\pm1)\lambda_j(n)n^{-1/2},
\]
if $n>0$ and $(n,q)=1$. Moreover, the reflection operator $R$ defined by $(Ru_j)(z)=u_j(-\bar{z})$
commutes with $\Delta$ and all $T_n$, so that we can also require
\begin{equation*}\label{eqn: Ru_j}
  Ru_j=\epsilon_j u_j.
\end{equation*}
Since $R$ is an involution, the space $\cC(\Gamma_0(q)\bs\bbH)$ is split into
even and odd cusp forms according to $\epsilon_j=1$ and $\epsilon_j=-1$.
We define
\begin{equation*}\label{eqn: omega_j}
  \omega_j := \frac{4\pi}{\cosh(\pi t_j)}|\rho_j(1)|^2.
\end{equation*}
By~\cite[Theorem 2]{iwaniec1990small}, we have
\begin{equation}\label{eqn: omega^*_j}
  \omega_j^*:=\frac{4\pi}{\cosh(\pi t_j)}\sum_{f\in\cV(u_j)}|\rho_f(1)|^2 \gg q^{-1}(qt_j)^{-\ve}.
\end{equation}

The Eisenstein series $E_\fa(z,s)$ is defined by
\begin{equation*}
  E_\fa(z,s) := \sum_{\gamma\in\Gamma_\fa\bs \Gamma_0(q)} \Im(\sigma_\fa^{-1}\gamma z)^s.
\end{equation*}
It has the following Fourier expansion
\[
  E_\fa(z,s) = \delta_\fa y^s+\varphi_\fa(s)y^{1-s}+\sum_{n\neq0}\varphi_\fa(n,s)W_s(nz),
\]
where $\delta_\fa=1$ if $\fa\sim\infty$, or $\delta_\fa=0$ otherwise.
Let
\begin{equation}\label{eqn: eta(n,s)}
  \eta(n,s) = \sum_{ad=|n|}\left(\frac{a}{d}\right)^{s-1/2}.
\end{equation}
The Eisenstein series $E_\fa(z,s)$ is even, and we have
\[
  T_n E_\fa(z,s) = \eta(n,s)E_\fa(z,s),
\]
if $(n,q)=1$.
Write $\eta_t(n)=\eta(n,1/2+it)$ and $E_{\fa,t}(z)=E_\fa(z,1/2+it)$. We define
\begin{equation}\label{eqn: omega(t)}
  \omega_\fa(t) := \frac{4\pi}{\cosh(\pi t)}|\varphi_\fa(1,1/2+it)|^2.
\end{equation}
And, by~\cite[p. 1188]{conrey2000cubic}, we have
\begin{equation}\label{eqn: omega^*(t)}
  \omega^*(t):=\sum_{\fa}\omega_\fa(t) \gg q^{-1-\ve}\min(|t|^{-\ve},|t|^2).
\end{equation}

Now we recall some background on Maass forms for $SL(3,\bbZ)$.
We will follow the notation in Goldfeld's book~\cite{goldfeld2006automorphic}.
Let $\phi$ be a Maass form of type $(\nu_1,\nu_2)$. We have the following
Fourier--Whittaker expansion
\begin{equation}\label{eqn: phi FE}
  \phi(z) = \sum_{\gamma\in U_2(\bbZ)\bs \Gamma_0(q)} \sum_{m_1=1}^{\infty}
            \sum_{m_2\neq0} \frac{A(m_1,m_2)}{m_1|m_2|}
            W_J\left(M\begin{pmatrix} \gamma & \\ & 1 \end{pmatrix}z,\nu_1,\nu_2,\psi_{1,1}\right),
\end{equation}
where $U_2(\bbZ)$ is the group of $2\times2$ upper triangular matrices
with integer entries and ones on the diagonal,
$W_J(z,\nu_1,\nu_2,\psi_{1,1})$ is the Jacquet--Whittaker function,
and $M=\diag(m_1|m_2|,m_1,1)$.
From now on, let $\phi$ be a self-dual Hecke--Maass form of type $(\nu,\nu)$ for $SL(3,\bbZ)$,
normalized to have the first Fourier coefficient $A(1,1)=1$.
For later purposes, we record the Hecke relation
\begin{equation}\label{eqn: HR}
  A(m,n) = \sum_{d|(m,n)}\mu(d)A\left(\frac{m}{d},1\right)A\left(1,\frac{n}{d}\right).
\end{equation}
Moreover, the Rankin--Selberg theory implies the bound
\begin{equation}\label{eqn: RS bound}
  \sum_{n\ll x}|A(1,n)|^2 \ll x,
\end{equation}
for all $x\geq1$.
We will also need the following estimate (see Blomer~\cite[Eq. (10) and (11)]{blomer2012subconvexity})
\begin{equation}\label{eqn: blomer}
  \sum_{n\leq x}|A(na,b)|^2 \ll x(ab)^{7/16+\ve},
  \quad \textrm{and}\quad
  \sum_{n\leq x}|A(na,b)| \ll x(ab)^{7/32+\ve}.
\end{equation}

\subsection{$L$-functions and the approximate functional equations}

The $L$-function attached to $\phi$ is $L(s,\phi)=\sum_{n=1}^{\infty}A(1,n)n^{-s}$,
and the completed $L$-function is given by
\[
  \Lambda(s,\phi) = \pi^{-3s/2}\prod_{j=1}^{3}\Gamma\left(\frac{s-\alpha_j}{2}\right) L(s,\phi),
\]
where $\alpha_1=3\nu-1$, $\alpha_2=0$, and $\alpha_3=1-3\nu$.
The $L$-function attached to the twist $\phi\times\chi$ is
\[
  L(s,\phi\times\chi)=\sum_{n=1}^{\infty}\frac{A(1,n)\chi(n)}{n^s},
\]
whose completed version is
\begin{equation*}\label{eqn: Lambda(s,phi.chi)}
  \Lambda(s,\phi\times\chi) = q^{3s/2}L_\infty(s,\phi\times\chi)L(s,\phi\times\chi),
\end{equation*}
where
\[
  L_\infty(s,\phi\times\chi) = \pi^{-3s/2}\prod_{j=1}^{3}\Gamma\left(\frac{s+\delta-\alpha_j}{2}\right),
\]
with $\delta=0$ or $1$ according to whether $\chi(-1)=1$ or $-1$.
Then $\Lambda(s,\phi\times\chi)$ is entire, and its functional equation is
\[
  \Lambda(s,\phi\times\chi)=\Lambda(1-s,\phi\times\chi).
\]
Note that the root number of $\Lambda(s,\phi\times\chi)$ is 1.

Next we consider the Rankin--Selberg convolution of $\phi$ with $u_j\times\chi$ given by
\[
  L(s,\phi\times u_j\times\chi) = \sum_{m,n=1}^{\infty}\frac{A(m,n)\lambda_j(n)\chi(n)}{(m^2n)^s}.
\]
The completed version of $L(s,\phi\times u_j\times\chi)$ is
\[
  \Lambda(s,\phi\times u_j\times\chi) = q^{3s} L_\infty(s,\phi\times u_j\times\chi)L(s,\phi\times u_j\times\chi),
\]
where
\[
  L_\infty(s,\phi\times u_j\times\chi) = \pi^{-3s} \prod_{\pm}\prod_{j=1}^{3}
  \Gamma\left(\frac{s+\delta\pm it_j-\alpha_j}{2}\right).
\]
The function $\Lambda(s,\phi\times u_j\times\chi)$ is entire, and its functional equation is
\[
  \Lambda(s,\phi\times u_j\times\chi)=\Lambda(1-s,\phi\times u_j\times\chi).
\]
Note that again the root number is 1.
Finally, we consider the convolution with the Eisenstein series which is defined as
\[
  L(s,\phi\times E_{\fa,t}\times\chi) = \sum_{m,n=1}^{\infty}\frac{A(m,n)\eta_t(n)\chi(n)}{(m^2n)^s}.
\]
By the definition of $\eta_t(n)$ \eqref{eqn: eta(n,s)}, by comparing the Euler products, we know that
\[
  L(s,\phi\times E_{\fa,t}\times\chi) = L(s+it,\phi\times\chi)L(s-it,\phi\times\chi).
\]

Now we consider the approximate functional equations for
$L(s,\phi\times u_j\times\chi)$ and $L(s,\phi\times E_{\fa,t}\times\chi)$.
We use the results from Blomer~\cite[\S2]{blomer2012subconvexity}.
Let
\begin{equation}\label{eqn: G(u)}
  G(u) = e^{u^2}.
\end{equation}
We have the following approximate functional equations, (see~\cite[Theorem 5.3]{iwaniec2004analytic}).
\begin{lemma}\label{lemma: AFE}
  We have
  \begin{equation}\label{eqn: AFE cusp}
    L(1/2,\phi\times u_j\times\chi) = 2\sum_{m,n=1}^{\infty}\frac{A(m,n)\lambda_j(n)\chi(n)}{(m^2n)^{1/2}}V_{t_j}\left(\frac{m^2n}{q^3}\right),
  \end{equation}
  where
  \[
    V_t(y) = \frac{1}{2\pi i}\int_{(3)}(\pi^3y)^{-u}\prod_\pm
    \prod_{j=1}^{3} \frac{\Gamma\left(\frac{1/2+u+\delta\pm it-\alpha_j}{2}\right)}
    {\Gamma\left(\frac{1/2+\delta\pm it-\alpha_j}{2}\right)} G(u)\frac{du}{u}.
  \]
  And similarly, we have
  \begin{equation}\label{eqn: AFE eis}
    L(1/2,\phi\times E_{\fa,t}\times\chi) = 2\sum_{m,n=1}^{\infty}\frac{A(m,n)\eta_t(n)\chi(n)}{(m^2n)^{1/2}}V_t\left(\frac{m^2n}{q^3}\right).
  \end{equation}
\end{lemma}

We see that $V_t(y)$ has the following properties which effectively limit the terms
in \eqref{eqn: AFE cusp} and \eqref{eqn: AFE eis} with $m^2n\ll (q(1+|t_j|))^{3+\ve}$
and $(q(1+|t|))^{3+\ve}$ respectively. Note that we can separate the variables $t$ and $y$ in $V_t(y)$
by the second part of the following lemma. Moreover, we see the $u$-integral can
be easily handled now. In our later application, we will take $U=\log^2(qT)$.

\begin{lemma}\label{lemma: V_t}
  \begin{itemize}
    \item [(i)]  We have
        \[
            y^k V_t^{(k)}(y) \ll \left(1+\frac{y}{(1+|t|)^3}\right)^{-A},
        \]
        and
        \[
            y^k V_t^{(k)}(y) = \delta_k + O\left(\left(\frac{y}{(1+|t|)^3}\right)^{\alpha}\right),
        \]
        where $\delta_0=1$, $\delta_k=0$ if $k>0$,
        and $0<\alpha\leq (1/2-|\Re(3\nu-1)|)/3$ (for example, we can take $\alpha=3/32$).
    \item [(ii)] For any $1<U\ll T^\ve$, $\ve>0$, and $|t-T|\ll T^{1-2\ve}$, we have the following approximation
        \begin{equation*}
          \begin{split}
             V_t(y) & =  \sum_{k=0}^{K/2}\sum_{l=0}^{K} T^{-2k} \left(\frac{t^2-T^2}{T^2}\right)^l
                            V_{k,l}\left(\frac{y}{T^3}\right)  + O\left(y^{-\ve}T^\ve e^{-U}\right) \\
                    & \hskip 150pt  + O\left(T^\varepsilon\left(\frac{1+|t-T|}{T}\right)^{K+1}\left(1+\frac{y}{T^3}\right)^{-A}\right),
          \end{split}
        \end{equation*}
        where
        \begin{equation}\label{eqn: V_kl}
          V_{k,l}(y) = \frac{1}{2\pi i} \int_{\ve-iU}^{\ve+iU} P_{k,l}(u)(2\pi)^{-3u} y^{-u} G(u)\frac{du}{u},
        \end{equation}
        for some polynomial $P_{k,l}$. Here $K$ is an even positive integer.
  \end{itemize}
\end{lemma}

\begin{proof}
  \begin{itemize}
    \item [(i)]  See Iwaniec--Kowalski~\cite[Proposition 5.4]{iwaniec2004analytic}.
    \item [(ii)] By Stirling's formula and contour shifts as in Iwaniec--Kowalski~\cite[p. 100]{iwaniec2004analytic}, we have (see Blomer~\cite[Lemma 1]{blomer2012subconvexity})
        \begin{equation*}
          V_t(y) = \frac{1}{2\pi i}\int_{\ve-iU}^{\ve+iU} \pi^{-3u}\prod_\pm
                    \prod_{j=1}^{3} \frac{\Gamma\left(\frac{1/2+u+\delta\pm it-\alpha_j}{2}\right)}
                    {\Gamma\left(\frac{1/2+\delta\pm it-\alpha_j}{2}\right)} y^{-u} G(u)\frac{du}{u}
                    + O\left(y^{-\ve}(1+|t|)^\ve e^{-U}\right).
        \end{equation*}
        The rest of the proof is following very closely to Young~\cite[\S5]{young2014weyl}.
        At first, by Stirling's formula, if $|\Im(z)|\rightarrow\infty$ (with fixed real part),
        but $|u|\ll |z|^{1/2}$, then
        \begin{equation*}\label{eqn: Gamma}
          \frac{\Gamma(z+u)}{\Gamma(z)} = z^u \left(1+\sum_{k=1}^{K}\frac{P_k(u)}{z^k}+O\left(\frac{(1+|u|)^{2K+2}}{|z|^{K+1}}\right) \right),
        \end{equation*}
        for certain polynomials $P_k(u)$ of degree $2k$.
        So for $|\Im(u)|\ll U$, and $t\asymp T$, 
        we have that
        \begin{equation*}
          \prod_\pm \prod_{j=1}^{3} \frac{\Gamma\left(\frac{1/2+u+\delta\pm it-\alpha_j}{2}\right)}
                    {\Gamma\left(\frac{1/2+\delta\pm it-\alpha_j}{2}\right)}
          = \left(\frac{t}{2}\right)^{3u} \left(1+\sum_{k=1}^{K/2}\frac{P_{2k}(u)}{t^{2k}}+O\left(\frac{(1+|u|)^{2K+2}}{t^{K+1}}\right) \right),
        \end{equation*}
        for a different collection of $P_k(u)$.
        Note that, in fact, the factor $\left(t/2\right)^{3u}$ is $\left((t/2)^2\right)^{3u/2}$,
        which is even as a function of $t$.
        For convenience, set $P_0(u)=1$.
        Hence
        \begin{equation}\label{eqn: V_t(y)=sum_k}
          \begin{split}
             V_t(y) & = \sum_{k=0}^{K/2}t^{-2k} \frac{1}{2\pi i}\int_{\ve-iU}^{\ve+iU}
             P_{2k}(u)\left(\frac{t}{2\pi}\right)^{3u} y^{-u}G(u)\frac{du}{u} \\
               & \hskip 80pt  + O\left(t^{-K-1+\varepsilon}\left(1+\frac{y}{t^3}\right)^{-A}\right)
               + O\left(y^{-\ve} t^\ve e^{-U}\right),
          \end{split}
        \end{equation}
        where the extra factor $\left(1+\frac{y}{t^3}\right)^{-A}$ arises
        from moving the contour to $\Re(u)=A$ if $y\geq t^3$, and to $\Re(u)=-1/4$
        if $y\leq t^3$ (here we use the fact $|\Re(\alpha_j)|\leq 7/32$).
        We further refine \eqref{eqn: V_t(y)=sum_k} by approximating $t$ by $T$.
        Since in our application $h(t)$ is very small unless $|t-T|\ll M\log^2 T$,
        where $M\ll T^{1/2}$ and $T$ large, our assumption $|t-T|\ll T^{1-2\ve}$ is
        flexible enough. Note that
        $\left|u\frac{t^2-T^2}{T^2}\right|\ll \left|u\frac{t-T}{T}\right| \ll T^{-\ve}$.
        By Taylor expansion, we have
        \begin{equation}\label{eqn: t^u}
          \begin{split}
            t^{3u} = T^{3u} e^{\frac{3u}{2}\log(1+\frac{t^2-T^2}{T^2})} & = T^{3u}\sum_{l=0}^{K} Q_l(u)\left(\frac{t^2-T^2}{T^2}\right)^l \\
                & \hskip 50pt  + O\left(T^\varepsilon(1+|u|)^{K+1}\left(\frac{|t-T|}{T}\right)^{K+1}\right),
          \end{split}
        \end{equation}
        for certain polynomial $Q_l(u)$ of degree $\leq l$. Similar results hold for $t^{3u-2k}$.
        So, by \eqref{eqn: V_t(y)=sum_k} and \eqref{eqn: t^u}, we prove this lemma. \qedhere
  \end{itemize}
\end{proof}

\subsection{The Kuznetsov formula for $\Gamma_0(q)$}

The two central tools we need in this paper are
the Kuznetsov formula for $\Gamma_0(q)$
and the Voronoi formula for $SL(3,\bbZ)$.
In this subsection, we recall the Kuznetsov formula,
and then in the next subsection, we will review
the Voronoi formula.
As usual, let
\[
  S(a,b;c) = {\sum_{d(c)}}^* e\left(\frac{ad+b\bar{d}}{c}\right)
\]
be the classical Kloosterman sum. For any $m,n\geq1$, and any test function
$h(t)$ which is even and satisfies the following conditions:
\begin{itemize}
  \item [(i)] $h(t)$ is holomorphic in $|\Im(t)|\leq 1/2+\ve$,
  \item [(ii)] $h(t)\ll (1+|t|)^{-2-\ve}$ in the above strip,
\end{itemize}
we have the following Kuznetsov formula (see~\cite[Eq. (3.17)]{conrey2000cubic} for example).
\begin{lemma}\label{lemma: KTF}
  For $m,n\geq1$, $(mn,q)=1$, we have
  \begin{equation*}
    \begin{split}
        & {\sum_j}'h(t_j)\omega_j^*\lambda_j(m)\lambda_j(n) + \frac{1}{4\pi}\int_{-\infty}^{\infty}h(t)\omega^*(t)\eta_t(m)\eta_t(n)dt \\
        & \hskip 120pt = \frac{1}{2}\delta_{m,n}H + \frac{1}{2}\sum_{q|c}\frac{1}{c}\sum_{\pm} S(n,\pm m;c)H^{\pm}\left(\frac{4\pi\sqrt{mn}}{c}\right),
    \end{split}
  \end{equation*}
  where $\sum'$ restricts to the even Hecke--Maass cusp forms, $\delta_{m,n}$ is the Kronecker symbol,
  \begin{equation}\label{eqn: H}
    \begin{split}
       H & = \frac{2}{\pi}\int_{0}^{\infty} h(t) \tanh(\pi t)t dt, \\
       H^+(x) & = 2i \int_{-\infty}^{\infty} J_{2it}(x)\frac{h(t)t}{\cosh(\pi t)}dt, \\
       H^-(x) & = \frac{4}{\pi} \int_{-\infty}^{\infty} K_{2it}(x)\sinh(\pi t)h(t)t dt,
    \end{split}
  \end{equation}
  and $J_\nu(x)$ and $K_\nu(x)$ are the standard $J$-Bessel function and $K$-Bessel function respectively.
\end{lemma}

\subsection{The Voronoi formula for $SL(3,\bbZ)$}

Let $\psi$ be a smooth compactly supported function on $(0,\infty)$,
and let $\tilde{\psi}(s):=\int_{0}^{\infty}\psi(x)x^s\frac{dx}{x}$
be its Mellin transform.
For $\sigma>7/32$, we define
\begin{equation}\label{eqn: Psi}
  \begin{split}
    \Psi^{\pm}(x) & := x\frac{1}{2\pi i} \int_{(\sigma)} (\pi^3x)^{-s} \prod_{j=1}^{3}
                                \frac{\Gamma\left(\frac{s+\alpha_j}{2}\right)}{\Gamma\left(\frac{1-s-\alpha_j}{2}\right)}\tilde{\psi}(1-s)ds \\
                  & \hskip 100pt \pm \frac{x}{i}\frac{1}{2\pi i} \int_{(\sigma)} (\pi^3x)^{-s} \prod_{j=1}^{3}
                                 \frac{\Gamma\left(\frac{1+s+\alpha_j}{2}\right)}{\Gamma\left(\frac{2-s-\alpha_j}{2}\right)}\tilde{\psi}(1-s)ds.
  \end{split}
\end{equation}
Here $\alpha_j$ has the same meaning as above, that is, $\alpha_1=3\nu-1,\alpha_2=0$, and $\alpha_3=1-3\nu$.
Note that changing $\psi(y)$ to $\psi(y/N)$ for a positive real number $N$ has the effect of
changing $\Psi^\pm(x)$ to $\Psi^\pm(xN)$.
The Voronoi formula on $GL(3)$ was first proved by Miller and Schmid~\cite{miller2006automorphic}.
The present version is due to Goldfeld and Li~\cite{goldfeld2006voronoi} with slightly renormalized variables
(see Blomer~\cite[Lemma 3]{blomer2012subconvexity}).
\begin{lemma}\label{lemma: VSF}
  Let $c,d,\bar{d}\in\bbZ$ with $c\neq0$, $(c,d)=1$, and $d\bar{d}\equiv1\pmod{c}$.
  Then we have
  \begin{equation*}
    \begin{split}
         \sum_{n=1}^{\infty} A(m,n)e\left(\frac{n\bar{d}}{c}\right)\psi(n)
         = \frac{c\pi^{3/2}}{2} \sum_{\pm} \sum_{n_1|cm} \sum_{n_2=1}^{\infty}
              \frac{A(n_2,n_1)}{n_1n_2} S\left(md,\pm n_2;\frac{mc}{n_1}\right)
              \Psi^{\pm}\left(\frac{n_1^2n_2}{c^3m}\right).
    \end{split}
  \end{equation*}
\end{lemma}

To prove Theorem \ref{thm: q} by applying Lemma \ref{lemma: VSF},
we need to know the asymptotic behaviour of $\Psi^{\pm}$.
This will be done in \S \ref{sec: Psi}.
Our work differs from Blomer~\cite{blomer2012subconvexity} in the nature of the weight function $\Psi^{\pm}$.
We will use the method of Young~\cite{young2014weyl}.
See Young~\cite[\S 8]{young2014weyl} for a more detail discussion of this method.
However, to prove Theorem \ref{thm: t}, we will need the following
asymptotic formula for $\Psi^\pm$.

\begin{lemma}\label{lemma: Psi=M+O}
  Suppose $\psi(y)$ is a smooth function, compactly supported on $[N,2N]$.
  Let $\Psi^\pm(x)$ be defined as in \eqref{eqn: Psi}.
  Then for any fixed integer $K\geq1$, and $xN\gg1$, we have
  \[
    \Psi^\pm(x) = x\int_0^\infty \psi(y) \sum_{\ell=1}^{K} \frac{\gamma_\ell}{(xy)^{\ell/3}}
    e\left(\pm3(xy)^{1/3}\right) dy + O\left((xN)^{1-K/3}\right),
  \]
  where $\gamma_\ell$ are constants depending only
  on $\alpha_1,\alpha_2,\alpha_3$, and $K$.
\end{lemma}

\begin{proof}
  See Li~\cite[Lemma 6.1]{li2009central} and Blomer~\cite[Lemma 6]{blomer2012subconvexity}.
\end{proof}

\subsection{The stationary phase lemma}

In this subsection, we will recall a result in McKee--Sun--Ye~\cite{mckee2015improved},
which will be used to prove Theorem \ref{thm: t}.
Let $f(x)$ be a real function, $2n+3$ times continuously differentiable for $\alpha\leq x\leq\beta$.
Suppose that $f'(x)$ changes signs only at $x=\gamma$, from negative to positive, with $\alpha<\gamma<\beta$.
Let $g(x)$ be a real function, $2n+1$ times continuously differentiable for $\alpha\leq x\leq\beta$.
Denote
\begin{equation}\label{eqn: H_i}
  H_1(x):=\frac{g(x)}{2\pi if^\prime(x)},\quad \textrm{and} \quad
  H_i(x):=-\frac{H_{i-1}^\prime(x)}{2\pi if^\prime(x)}
  \quad \textrm{for $i\geq2$.}
\end{equation}
Let
\begin{equation}\label{eqn: lambda_k}
  \lambda_k := \frac{f^{(k)}(\gamma)}{k!}, \quad \textrm{for}\quad k=2,\ldots,2n+2,
\end{equation}
\begin{equation}\label{eqn: eta_k}
  \eta_k := \frac{g^{(k)}(\gamma)}{k!}, \quad \textrm{for}\quad k=0,\ldots,2n,
\end{equation}
and
\begin{equation}\label{eqn: varpi_k}
  \varpi_k = \eta_k + \sum_{\ell=0}^{k-1}\eta_\ell
  \sum_{j=1}^{k-\ell}\frac{C_{k\ell j}}{\lambda_{2}^j}
  \sum_{\substack{3\leq n_1,\ldots,n_j\leq 2n+3 \\ n_1+\cdots+n_j=k-\ell+2j}}
  \lambda_{n_1}\cdots\lambda_{n_j},
\end{equation}
for some constant coefficients $C_{k\ell j}$.
See~\cite[\S2]{mckee2015improved} for more details.

\begin{lemma}\label{lemma: MSY}
  Let $f(x)$, $g(x)$, and $H_k(x)$ be defined as above.
  Suppose that there are positive parameters $M_0$, $N_0$, $T_0$, $U_0$, with
  \begin{equation}\label{eqn: M0}
    M_0 > \beta-\alpha,
  \end{equation}
  and positive constants $C_r$ such that for $\alpha\leq x\leq\beta$,
  \begin{equation}\label{eqn: f^(r)}
    |f^{(r)}(x)|\leq C_r \frac{T_0}{M_0^r},\quad  \text{for}\quad r=2,3,...,2n+3,
  \end{equation}
  \begin{equation}\label{eqn: f''}
    |f^{\prime\prime}(x)| \geq \frac{T_0}{C_2M_0^2},
  \end{equation}
  and
  \begin{equation}\label{eqn: g^(s)}
    |g^{(s)}(x)|\leq C_s\frac{U_0}{N_0^s},\quad \text{for}\quad s=0,1,2,...,2n+1.
  \end{equation}
  If $T_0$ is sufficiently large comparing to the constants $C_r$,
  we have for $n\geq2$ that
  \begin{equation}\label{eqn: MSY}
    \begin{split}
        & \int_{\alpha}^{\beta}g(x)e(f(x))dx
        = \frac{e\Big(f(\gamma)+\frac{1}{8}\Big)}{\sqrt{f''(\gamma)}}
            \Big(g(\gamma)+\sum_{j=1}^{n}\varpi_{2j}\frac{(-1)^{j}(2j-1)!!}{(4\pi i\lambda_2)^j}\Big)
           \\
        & \hskip 60pt  + \Big[e(f(x))\cdot\sum_{i=1}^{n+1}H_{i}(x)\Big]_{\alpha}^{\beta}
        + O\left(\frac{U_0M_0^{2n+5}}{T_0^{n+2}N_0^{n+2}}\Big(\frac{1}{(\gamma-\alpha)^{n+2}}+\frac{1}{(\beta-\gamma)^{n+2}}\Big)\right) \\
        & \hskip 60pt + O\left(\frac{U_0M_0^{2n+4}}{T_0^{n+2}}
\Big(\frac{1}{(\gamma-\alpha)^{2n+3}}+\frac{1}{(\beta-\gamma)^{2n+3}}\Big)\right)\\
		& \hskip 60pt + O\left(\frac{U_0M_0^{2n+4}}{T_0^{n+2}N_0^{2n}}\Big(\frac{1}{(\gamma-\alpha)^{3}}+\frac{1}{(\beta-\gamma)^{3}}\Big)\right) 
        + O\left(\frac{U_0}{T_0^{n+1}}\Big(\frac{M_0^{2n+2}}{N_0^{2n+1}}+M_0\Big)\right).
    \end{split}
  \end{equation}
\end{lemma}

\begin{proof}
  See McKee--Sun--Ye~\cite[Proposition 2.2]{mckee2015improved}.
\end{proof}

\section{Initial setup of Theorem \ref{thm: q}}\label{sec: setup q}


We are now ready to start with the proof of Theorem \ref{thm: q}.
As indicated in the introduction,
both results follow rather easily from the following bound.
\begin{proposition}\label{prop: q}
  With notation as above, for any $\ve>0$, $T$ large, and $M\asymp T^{1/2}$,
  we have
  \[
    \sum_{\substack{u_j\in\cB^*(q)\\ T-M\leq t_j\leq T+M}} L(1/2,\phi\times u_j\times\chi)
            + \frac{1}{4\pi}\int_{T-M}^{T+M}|L(1/2+it,\phi\times\chi)|^2dt
    \ll_{\phi,\ve} q^{5/4}TM(qT)^{\ve}.
  \]
\end{proposition}

Theorem \ref{thm: q} follows from the above proposition.
The key ingredient is Lapid's theorem~\cite{lapid2003nonnegativity}
about the nonnegativity of $L(1/2,\phi\times u_j\times\chi)$.
See Blomer~\cite[\S4]{blomer2012subconvexity} for more details.

To prove Proposition \ref{prop: q},
we introduce the spectrally normalized first moment of the central values of $L$-functions
\begin{equation}\label{eqn: cW}
  \cM := \sum_{u_j\in\cB^*(q)}h(t_j)\omega_j^* L(1/2,\phi\times u_j\times\chi)
  + \frac{1}{4\pi}\int_{-\infty}^{\infty}h(t)\omega^*(t)|L(1/2+it,\phi\times\chi)|^2dt,
\end{equation}
where
\begin{equation}\label{eqn: f(t)}
  h(t) := \frac{1}{\cosh\left(\frac{t-T}{M}\right)} + \frac{1}{\cosh\left(\frac{t+T}{M}\right)}.
\end{equation}
Here we choose the above weight because we can use Young's results~\cite{young2014weyl} directly.
However, maybe the following weight function of Li
\begin{equation*}
  h(t) = e^{-\frac{(t-T)^2}{M^2}} + e^{-\frac{(t+T)^2}{M^2}}
\end{equation*}
will work too. By \eqref{eqn: omega^*_j} and \eqref{eqn: omega^*(t)}, we have
\[
  \sum_{\substack{u_j\in\cB^*(q)\\ T-M\leq t_j\leq T+M}} L(1/2,\phi\times u_j\times\chi)
            + \frac{1}{4\pi}\int_{T-M}^{T+M}|L(1/2+it,\phi\times\chi)|^2dt \\
  \ll \cM q^{1+\ve}T^\ve,
\]
for any $\ve>0$.
Therefore, to prove Proposition \ref{prop: q}, we just need to prove
\begin{equation}\label{eqn: cW<<}
  \cM \ll_{\phi,\ve} q^{1/4}TM(qT)^\ve.
\end{equation}

Applying Lemma \ref{lemma: AFE} to $\cM$, we have
\begin{equation*}
  \begin{split}
      \cM & = 2\sum_{u_j\in\cB^*(q)}h(t_j)\omega_j^* \sum_{m,n=1}^{\infty}\frac{A(m,n)\lambda_j(n)\chi(n)}{(m^2n)^{1/2}}V_{t_j}\left(\frac{m^2n}{q^3}\right) \\
       & \hskip 90pt + 2\frac{1}{4\pi}\int_{-\infty}^{\infty}h(t)\omega^*(t)
         \sum_{m,n=1}^{\infty}\frac{A(m,n)\eta_t(n)\chi(n)}{(m^2n)^{1/2}}V_{t}\left(\frac{m^2n}{q^3}\right)dt.
  \end{split}
\end{equation*}
By Lemma \ref{lemma: V_t}, we can truncate the $m,n$-sums at
\[
  m^2n\leq (qT)^{3+\ve}
\]
at the cost of a negligible error.
Now we handle the weight $V_t(y)$. By Lemma \ref{lemma: V_t} (we choose $U=\log^2(qT)$),
to prove \eqref{eqn: cW<<}, we need to prove
\[
  \begin{split}
      & 2\sum_{m^2n\leq (qT)^{3+\ve}} \frac{A(m,n)\chi(n)}{(m^2n)^{1/2+u}}
            \bigg(\sum_{u_j\in\cB^*(q)}h_{k,l}(t_j)\omega_j^* \lambda_j(n) \\
       & \hskip 150pt  + \frac{1}{4\pi}\int_{-\infty}^{\infty}h_{k,l}(t)\omega^*(t)\eta_t(n)dt \bigg)
       \ll_{\phi,\ve} q^{1/4}TM(qT)^\ve,
  \end{split}
\]
uniformly in $u\in[\ve-i\log^2(qT),\ve+i\log^2(qT)]$, where
\begin{equation}\label{eqn: h_kl}
  h_{k,l}(t) = T^{-2k-2l} \left(t^2-T^2\right)^l h(t).
\end{equation}
Now we apply the Kuznetsov formula with $m=1$
(note that $m$ has different meaning here), we arrive at bounding
\[
  \sum_{m^2n\leq (qT)^{3+\ve}}\frac{A(m,n)\chi(n)}{(m^2n)^{1/2+u}}
  \left(\delta_{n,1}H + \sum_{q|c}\frac{1}{c}\sum_{\pm}S(n,\pm1;c)H^{\pm}\left(\frac{4\pi \sqrt{n}}{c}\right)\right),
\]
where $H,\ H^\pm$ are defined as in \eqref{eqn: H} with $h(t)=h_{k,l}(t)$.
We will only deal with the case $k=l=0$, since the others can be handled similarly.
By \eqref{eqn: RS bound}, and the fact $H\ll TMT^\ve$,
we know the diagonal term is bounded by
\begin{equation}\label{eqn: cD}
  \cD := \sum_{m^2n\leq (qT)^{3+\ve}}\frac{A(m,n)\chi(n)}{(m^2n)^{1/2+u}} \delta_{n,1}H
      \ll \sum_{m^2\leq (qT)^{3+\ve}}\frac{|A(m,1)|}{m}|H| \ll TM (qT)^\ve.
\end{equation}

Now we need to bound the off-diagonal terms
\begin{equation}\label{eqn: cR}
  \cR^{\pm} := \sum_{m^2n\leq (qT)^{3+\ve}}\frac{A(m,n)\chi(n)}{(m^2n)^{1/2+u}}
  \sum_{q|c}\frac{1}{c}S(n,\pm1;c)H^{\pm}\left(\frac{4\pi \sqrt{n}}{c}\right).
\end{equation}
By the argument in Blomer~\cite[\S5]{blomer2012subconvexity}, it is then enough to show that
\begin{equation}\label{eqn: sum over cS}
  \sum_{\substack{m^2\delta^3\leq (qT)^{3+\ve}\\ (\delta,q)=1,\ |\mu(\delta)|=1}}
  \frac{|A(1,m)|}{m\delta^{3/2}}  \sup_N \frac{1}{N^{1/2}} |\cS_\sigma(q,N;\delta)|
  \ll  q^{1/4}TM(qT)^\ve,
\end{equation}
for $\sigma\in\{\pm1\}$, where
\begin{equation}\label{eqn: cS}
  \cS_\sigma(q,N;\delta) :=
  \sum_{q|c}\frac{1}{c} \sum_{n}A(n,1)\chi(n) S(\delta n,\sigma;c)
  \psi_{\sigma}\left(\frac{n}{N};\frac{\sqrt{\delta N}}{c}\right),
\end{equation}
with
\begin{equation}\label{eqn: psi_s}
  \psi_\sigma(y;D) := \left\{\begin{array}{ll}
                           w(y)y^{-u}H^+(4\pi \sqrt{y}D), & \textrm{if } \sigma=1, \\
                           w(y)y^{-u}H^-(4\pi \sqrt{y}D), & \textrm{if } \sigma=-1,
                         \end{array} \right.
\end{equation}
$w$ a suitable fixed smooth function with support in $[1,2]$, and
\begin{equation}\label{eqn: N}
  N \leq \frac{(qT)^{3+\ve}}{m^2\delta^3}.
\end{equation}
Here we suppress the dependence on $u$ in $\psi_\sigma(y;D)$.
As Blomer~\cite[\S5]{blomer2012subconvexity} did, by the Voronoi formula, we have
\begin{equation}\label{eqn: cS_s}
  \begin{split}
     &  \cS_\sigma(q,N;\delta) = \frac{\pi^{3/2}}{2} \sum_{\pm} \sum_{q|c}\frac{1}{c^2}
            \sum_{c_1|c} c_1  \sum_{n_1|c_1}\sum_{n_2=1}^{\infty}\frac{A(n_2,n_1)}{n_1n_2}
        \Psi_{\sigma}^{\pm}\left(\frac{n_1^2n_2N}{c_1^3};\frac{\sqrt{\delta N}}{c}\right)
            \cT_{\delta,c_1,n_1,n_2}^{\pm,\sigma}(c,q),
  \end{split}
\end{equation}
where
\begin{equation}\label{eqn: cT}
  \begin{split}
    \cT_{\delta,c_1,n_1,n_2}^{\pm,\sigma}(c,q) & =
        {\sum_{b(c_1)}}^* {\sum_{d(c)}}^{*}e\left(\sigma\frac{\bar{d}}{c}\right)
        \sum_{a(c)}\chi(a)e\left(\frac{-\bar{b}a}{c_1}\right)
         e\left(\frac{\delta da}{c}\right)
            \underset{f(c_1/n_1)}{{\sum}^*}e\left(\frac{bf\pm n_2\bar{f}}{c_1/n_1}\right),
  \end{split}
\end{equation}
and $\Psi_{\sigma}^{\pm}(x;D)$ is defined as in \eqref{eqn: Psi} with $\psi(x)=\psi_\sigma(x;D)$.

We end this section by truncating the $c$-sum.
Define
\begin{equation*}
     \cS_\sigma(q,N;C;\delta)
     = \sum_{\substack{c\asymp C \\ q|c}}\frac{1}{c} \sum_{n}A(n,1)\chi(n) S(\delta n,\sigma;c)
  \psi_{\sigma}\left(\frac{n}{N};\frac{\sqrt{\delta N}}{c}\right).
\end{equation*}
Using the weak bound $H^\pm(y)\ll Ty^{3/4}$, and the Weil bound for Kloosterman sums,
we have
\[
  \cS_\sigma(q,N;C;\delta)
  \ll T (qT)^{33/8+\ve} C^{-1/4+\ve},
\]
which is good enough if $C$ is a large power of $qT$.
Therefore, it suffices to bound $\cS$ with $C\ll (qT)^B$
for some large but fixed $B$.

\section{Analytic separation of variables}\label{sec: Psi}

Our goal in this section is to handle $\Psi_{\sigma}^{\pm}(x;D)$.
We follow the approach in Young~\cite{young2014weyl}.
The following argument will need the stationary phase method.
We'll use the following lemma (see~\cite[Lemma 8.1 and Proposition 8.2]{blomer2013distribution}).
\begin{lemma}\label{lemma: SP}
  Suppose that $w$ is a smooth weight function with compact support on $[X, 2X]$,
  satisfying $w^{(j)}(t) \ll X^{-j}$, for $X \gg 1$
  (in particular, $w$ is inert with uniformity in $X$) and $j\geq0$.
  Also suppose that $\phi$ is smooth and satisfies $\phi^{(j)}(t) \ll \frac{Y}{X^j}$
  for some $Y \gg X^{\varepsilon}$ and $j\geq1$.  Let
  \begin{equation}
     I = \int_{-\infty}^{\infty} w(t) e^{i \phi(t)} dt.
  \end{equation}
  \begin{enumerate}
    \item If $\phi'(t) \gg \frac{Y}{X}$ for all $t$ in the support of $w$,
        then $I \ll_A Y^{-A}$ for $A$ arbitrarily large.
    \item If $\phi''(t) \gg \frac{Y}{X^2}$ for all $t$ in the support of $w$,
        and there exists $t_0 \in \bbR$ such that $\phi'(t_0) = 0$ (note $t_0$ is necessarily unique), then
        \begin{equation}
          I = \frac{e^{i \phi(t_0)}}{\sqrt{\phi''(t_0)}} F(t_0) + O(Y^{-A}),
        \end{equation}
        where $F$ is an inert function (depending on $A$, but uniformly in $X$ and $Y$) supported on $t_0 \asymp X$.
  \end{enumerate}
\end{lemma}

Here, following Young~\cite{young2014weyl}, we say a smooth function
$f(x_1,\ldots,x_n)$ on $\bbR^n$ \emph{inert} if
\begin{equation}\label{eqn: inert}
  x_1^{k_1}\cdots x_n^{k_n} f^{(k_1,\ldots,k_n)}(x_1,\ldots,x_n) \ll 1,
\end{equation}
with an implied constant depending on $k_1,\ldots,k_n$ and with the
superscript denoting partial differentiation.
Now, we recall some results about $H^\pm$ from Young~\cite[\S7]{young2014weyl}.

\begin{lemma}\label{lemma: H+}
  Let $H^+$ be given by \eqref{eqn: H}.
  There exists a function $g$ depending on $T$ and $M$ satisfying $g^{(j)}(y)\ll_{j,A} (1+|y|)^{-A}$,
  so that
  \begin{equation}\label{eqn: H^+=int}
    H^+(x) = MT\int_{|v|\leq \frac{M^\ve}{M}}\cos(x\cosh(v))e\left(\frac{vT}{\pi}\right)g(Mv)dv + O(T^{-A}).
  \end{equation}
  Furthermore, $H^+(x)\ll T^{-A}$ unless $x\gg MT^{1-\ve}$.
  And if $x\gg MT^{1-\ve}$, then $H^+(x)\ll TMx^{-1/2}$.
\end{lemma}

\begin{proof}
  See Young~\cite[Lemma 7.1]{young2014weyl}.
  And the upper bound for $H^+$ when $x\gg MT^{1-\ve}$ comes from \eqref{eqn: H^+=int} and Lemma \ref{lemma: SP}.
\end{proof}

\begin{lemma}\label{lemma: H-}
  Let $H^-$ be given by \eqref{eqn: H}.
  There exists a function $g$ depending on $T$ and $M$ satisfying $g^{(j)}(y)\ll_{j,A} (1+|y|)^{-A}$,
  so that
  \begin{equation}\label{eqn: H^-=}
    H^-(x) = MT\int_{|v|\leq \frac{M^\ve}{M}}\cos(x\sinh(v))e\left(\frac{vT}{\pi}\right)g(Mv)dv + O(T^{-A}).
  \end{equation}
  Furthermore, $H^-(x)\ll (x+T)^{-A}$ unless $x\asymp T$.
  And if $x\asymp T$, then we have $H^{-}(x)\ll T^{1+\ve}$.
\end{lemma}

\begin{proof}
  See Young~\cite[Lemma 7.2]{young2014weyl}.
  And the upper bound for $H^-$ when $x\asymp T$ is an easy consequence of \eqref{eqn: H^-=}.
\end{proof}

Define
\begin{equation}\label{eqn: tilde(Psi)}
  \check{\Psi}_\sigma^\pm(x;D) := e\left(\mp\sigma\frac{x}{D^2}\right)\Psi_\sigma^\pm(x;D),
\end{equation}
and
\begin{equation}\label{eqn: Upsilon}
  \Upsilon(t)=\Upsilon_{X,D}(t):=\int_{0}^{\infty}w\left(\frac{x}{X}\right)\check{\Psi}_\sigma^\pm(x;D)x^{it}\frac{dx}{x},
\end{equation}
where $w$ is a fixed smooth function supported on $[1/2,3]$, and with value $1$ on $[1,2]$.
(Note that this $w$ differs from that of \eqref{eqn: psi_s}.)
Now together with the Mellin technique, we can prove the following lemma, which will help us to separate the variables.

\begin{lemma}\label{lemma: Psi}
  Let $x\asymp X$ with $X \gg (qT)^{-B}$ for some large but fixed $B$.
  \begin{enumerate}
    \item[(i)] We have $\check{\Psi}_\sigma^\pm(x;D)\ll T^{-A}$ unless 
        \begin{equation}\label{eqn: D}
          \left\{
            \begin{array}{ll}
                D\gg TM^{1-\ve}, & \textrm{if } \sigma=1, \\
                D\asymp T, & \textrm{if } \sigma=-1.
            \end{array}
          \right.
        \end{equation}
    \item[(ii)] 
        When $X\ll T^\ve$, if $D$ satisfies \eqref{eqn: D}, then we have
            \begin{equation}\label{eqn: dPsi}
              x^k\frac{d^k}{dx^k}\check{\Psi}_\sigma^\pm(x;D)\ll_{k,\ve} \left\{\begin{array}{ll}
                                               X^{2/3}TMD^{-1/2}T^{\ve}, & \textrm{if } \sigma=1, \\
                                               X^{2/3}T^{1+\ve}, & \textrm{if } \sigma=-1.
                                             \end{array}\right.
            \end{equation}
            Note that here the $\ve$ on the right hand side may depend on $k$.
            Furthermore, if $x\in[X,2X]$, we have
            \begin{equation}\label{eqn: Psi x<}
              \check{\Psi}_\sigma^\pm(x;D) = \frac{1}{2\pi}\int_{-T^{\epsilon}}^{T^{\epsilon}} \Upsilon(t)x^{-it}dt + O(T^{-A}).
            \end{equation}
            And for $|t|\ll T^{\epsilon}$, we have
            \begin{equation}\label{eqn: Upsilon<<}
              \Upsilon(t) \ll \left\{\begin{array}{ll}
                                               X^{2/3}TMD^{-1/2}T^{\ve}, & \textrm{if } \sigma=1, \\
                                               X^{2/3}T^{1+\ve}, & \textrm{if } \sigma=-1.
                                             \end{array}\right.
            \end{equation}
    \item[(iii)] When $X\gg T^\ve$, we have
        \begin{equation}\label{eqn: Psi_s^pm}
            \check{\Psi}_\sigma^\pm(x;D) = \sum_{\ell=1}^{K} \gamma_\ell \frac{x^{5/6}M}{x^{\ell/3}}
            L(x;D) + O(T^{-A}),
        \end{equation}
        where $L$ is a function that takes the form
        \begin{equation}\label{eqn: L}
            L(x;D) = \int_{|t|\ll U} \lambda_{X,T}(t) \left(\frac{x}{D^2}\right)^{it}dt
        \end{equation}
        with the following parameters.
        Here $\lambda_{X,T}(t)\ll1$ does not depend on $x$ and $D$.
        If $\sigma=1$, then $U=T^2/D$; furthermore, $L$ vanishes unless
        \begin{equation}\label{eqn: X&D+}
            X\asymp D^3, \quad \textrm{and} \quad D\gg MT^{1-\ve}.
        \end{equation}
        If $\sigma=-1$, then $U=T^{2/3}X^{1/3}D^{-2/3}$; in addition, $L$ vanishes unless
        \begin{equation}\label{eqn: X&D-}
            X \ll D^3 M^{\ve-3}, \quad \textrm{and} \quad D\asymp T.
        \end{equation}
  \end{enumerate}
\end{lemma}

\begin{proof}
  We first handle the case $X\gg T^\ve$.
  By Blomer~\cite[Lemma 6]{blomer2012subconvexity}, we have
  \begin{equation}\label{eqn: Psi= Blomer}
    \Psi_\sigma^{\pm}(x;D) = x \int_{0}^{\infty} \psi_\sigma(y;D) \sum_{\ell=1}^{K}
    \frac{\gamma_\ell}{(xy)^{\ell/3}}e\left(\pm3(xy)^{1/3}\right)dy + O(T^{-A}),
  \end{equation}
  for some constants $\gamma_\ell$ depending only on $\alpha_1,\alpha_2,\alpha_3$.
  Recall the definition of $\psi_\sigma(y;D)$ \eqref{eqn: psi_s}.
  By Lemmas \ref{lemma: H+} and \ref{lemma: H-}, we arrive at
  \[
    \sum_{\ell=1}^{K} \frac{xMT\gamma_\ell}{x^{\ell/3}}
    \int_{|v|\ll \frac{M^{\ve}}{M}} g(Mv)e\left(\frac{vT}{\pi}\right)
    \left(\int_{0}^{\infty} w(y)y^{-u} e\left(2\sqrt{y}D\phi_\sigma(v)
    \pm 3(xy)^{1/3}\right) y^{-\ell/3} dy \right) dv,
  \]
  where $\phi_\sigma(v)=\pm\cosh(v)$ for $\sigma=1$,
  and $\phi_\sigma(v)=\pm\sinh(v)$ for $\sigma=-1$.
  The $\pm$ in $\phi_\sigma(v)$ is different from that of $\Psi_\sigma^{\pm}$.

  The $y$-integral can be analyzed by stationary phase. By Lemma \ref{lemma: SP}, we know
  the above integral is small unless a stationary point exists, which implies
  \begin{equation}\label{eqn: phi(v)}
    |\phi_\sigma(v)|=\pm \phi_\sigma(v) \asymp  X^{1/3}/D.
  \end{equation}
  In addition, since $|v|\ll M^\ve/M$, we have $X\asymp D^3$ if $\sigma=1$,
  and $X\ll D^3M^{\ve-3}$ if $\sigma=-1$.

  At this point, we can restrict the size of $X$. Recall that in our application,
  $D=\sqrt{\delta N}/c$ and $N$ satisfying \eqref{eqn: N}, we have $D\ll (qT)^2$.
  Hence we can assume that $X\ll (qT)^6$.
  Otherwise, we get $\Psi_\sigma^{\pm}(x;D)\ll T^{-A}$.

  Now we consider the range of $D$ that makes $\Psi_\sigma^\pm(x;D)$ non-negligible.
  At first, by the above argument, we can restrict ourself to the case $(qT)^{-B}\ll X\ll (qT)^B$.
  By \eqref{eqn: Psi} and Parseval's formula, we have
  \begin{equation}\label{eqn: Psi Parseval}
    \Psi_\sigma^\pm(x;D) = x\int_{0}^{\infty} \psi_\sigma(y;D) g^{\pm}(\pi^3 xy)dy,
  \end{equation}
  where
  \[
    g^{\pm}(y) = \frac{1}{2\pi i} \int_{(c)} G^{\pm}(s)y^{-s}ds
  \]
  is the inverse Mellin transform of
  \[
    G^{\pm}(s) = \prod_{j=1}^{3} \frac{\Gamma\left(\frac{s+\alpha_j}{2}\right)}{\Gamma\left(\frac{1-s-\alpha_j}{2}\right)}
                        \pm \frac{1}{i} \prod_{j=1}^{3}                           \frac{\Gamma\left(\frac{1+s+\alpha_j}{2}\right)}{\Gamma\left(\frac{2-s-\alpha_j}{2}\right)}.
  \]
  Now, by Lemmas \ref{lemma: H+} and \ref{lemma: H-}, we can assume that
  $D\gg MT^{1-\ve}$ if $\sigma=1$, and $D\asymp T$ if $\sigma=-1$.
  Otherwise, we have $\Psi_\sigma^{\pm}(x;D)\ll T^{-A}$.
  Thus we give the proof of part (i).

  Assuming \eqref{eqn: phi(v)}, the stationary point is at $y_0=x^2(\phi_\sigma(v)D)^{-6}\asymp1$,
  so, by Lemma \ref{lemma: SP}, we have
  \[
    \begin{split}
      \int_{0}^{\infty} w(y)y^{-u} e\left(2\sqrt{y}D\phi_\sigma(v) \pm 3(xy)^{1/3}\right) y^{-\ell/3} dy  = x^{-1/6} e\left(\mp \frac{x}{\phi_\sigma^2(v)D^2}\right) w_1(v) + O(T^{-A}),
    \end{split}
  \]
  where $w_1$ is inert in terms of $v$, and $w_1$ has support on \eqref{eqn: phi(v)}.
  The fact that $w_1$ is inert in terms of $v$ needs some discussion.
  We naturally obtain an inert function in terms of $\phi_\sigma(v)$,
  but since $\phi_\sigma(v)$ has bounded derivatives for $|v|\leq1$,
  we do get an inert function of $v$.
  Hence, to bound $\Psi_\sigma^{\pm}(x;D)$, we only need to estimate
  \begin{equation}
      \sum_{\ell=1}^{K} \frac{x^{5/6}MT\gamma_\ell}{x^{\ell/3}} e\left(\mp\sigma\frac{x}{D^2}\right) \Phi_\sigma\left(\frac{x}{D^2}\right),
  \end{equation}
  where
  \begin{equation}\label{eqn: Phi}
     \Phi_\sigma(y) = \int_{|v|\ll \frac{M^{\ve}}{M}} g(Mv)e\left(\frac{vT}{\pi}\right) e\left(\pm y(\sigma-\phi_\sigma^{-2}(v))\right) w_1(v) dv.
  \end{equation}
  Finally, we shall use the Mellin technique to analyze $\Phi_\sigma(y)$.
  By the same proof as Young~\cite[Lemma 8.2]{young2014weyl} did, for
  \[
    y\asymp Y\asymp X/D^2 \asymp \left\{\begin{array}{ll}
               X^{1/3}, & \textrm{if } \sigma=1, \\
               XT^{-2}, & \textrm{if } \sigma=-1,
             \end{array}\right.
  \]
  we have
  \begin{equation}\label{eqn: Phi=}
      \Phi_\sigma(y) = \frac{1}{T} \int_{|t|\ll U} \lambda_{Y,T}(t)y^{it}dt + O(T^{-A}),
  \end{equation}
  where $\lambda_{Y,T}$ and $U$ depend on $Y,T$. Precisely, we have $\lambda_{Y,T}(t)\ll 1$, and
  \begin{equation}\label{eqn: U}
      \left\{\begin{array}{ll}
               U=T^2/Y, & \textrm{if } \sigma=1, \\
               U=Y^{1/3}T^{2/3}, & \textrm{if } \sigma=-1.
             \end{array}\right.
  \end{equation}
  (Note that the assumption $Y\gg1$ in Young~\cite[Lemma 8.2]{young2014weyl} is not used.
  Here we just need $Y/|v_0|^2\gg T^{\ve}$, where $|v_0|=x^{1/3}/D$ in the proof.
  And we can derive this from $Y/|v_0|^2 \asymp X/(D|v_0|)^2\asymp X^{1/3}\gg T^{\ve}$.)
  Note that in either case, we have $U\gg T^\ve$.
  Now, by \eqref{eqn: phi(v)}, we have $Y\asymp D$ if $\sigma=1$.
  This proves part (iii).

  Finally, we deal with the case $X\ll T^\ve$.
  By Blomer~\cite[Lemma 7]{blomer2012subconvexity}, for $D$ satisfying \eqref{eqn: D}, we have
  \begin{equation}\label{eqn: blomer Psi}
    x^k\frac{d^k}{dx^k}\check{\Psi}_\sigma^{\pm}(x;D) \ll_k (1+x^{1/3})^k x^{2/3}\|\psi_\sigma\|_\infty.
  \end{equation}
  Now, by \eqref{eqn: psi_s} and Lemmas \ref{lemma: H+} and \ref{lemma: H-},
  we prove the upper bound \eqref{eqn: dPsi}.
  Next, we want to use the Mellin technique to separate the variables.
  Recall that
  \[
    \Upsilon(t)=\int_{0}^{\infty}w\left(\frac{x}{X}\right)\check{\Psi}_\sigma^\pm(x;D)x^{it}\frac{dx}{x}.
  \]
  Note that for $|t|\gg T^\epsilon$, (taking $\epsilon>2\ve$), we have $t/x \gg T^\ve$.
  So using integral by parts many times, for $|t|\gg T^\epsilon$, we have $\Upsilon(t)\ll (tT)^{-A}$.
  By the Mellin inversion, for $x\in[X,2X]$, we have
  \[
      \check{\Psi}_\sigma^\pm(x;D)
    = \frac{1}{2\pi}\int_{-\infty}^{\infty} \Upsilon(t) x^{-it}dt
    = \frac{1}{2\pi}\int_{-T^{\epsilon}}^{T^{\epsilon}} \Upsilon(t)x^{-it}dt + O(T^{-A}).
  \]
  And for $|t|\ll T^{\epsilon}$, the upper bound \eqref{eqn: Upsilon<<} of $\Upsilon(t)$ is a consequence of
  \eqref{eqn: Upsilon} and \eqref{eqn: blomer Psi}.
  Thus we prove part (ii). This finishes the proof of the lemma.
\end{proof}

Lemma \ref{lemma: Psi} is good enough to give a nice bound for
the terms related to the $K$-Bessel function. However, we don't
know how to apply both the large sieve inequalities and a second use
of Voronoi formula when we want to bound the terms related
to the $J$-Bessel function. So, on the one hand, in the following
sections, we will get a bound without using the Voronoi formula
twice. This result is good in the $q$-aspect and not too bad in the $t$-aspect.
And then, on the other hand, in \S \ref{sec: Psi II}, we will use another
method to deal with the integral transforms that appear on the
right hand side of the Voronoi formula. This will be done by following
Blomer~\cite[\S3]{blomer2012subconvexity}, Li~\cite[\S4]{li2011bounds},
and McKee--Sun--Ye~\cite[\S6]{mckee2015improved}.
By doing this, we will obtain a bound which is good in the $t$-aspect,
and not too bad in the $q$-aspect. Then combining these two bounds,
one can get a hybrid subconvexity bound.

\section{Applying the large sieve} \label{sec: large_sieve}

Let
\begin{equation}\label{eqn: H(w;q)}
  H(w;q) = \sum_{u,v(\mod q)} \chi_q(uv(u+1)(v+1))e_q((uv-1)w).
\end{equation}
By Conrey--Iwaniec~\cite[Eq. (11.7)]{conrey2000cubic}, we have
\begin{equation}\label{eqn: H to H^*}
  H(w;q) = \sum_{q_1q_2=q} \mu(q_1)\chi_{q_1}(-1)H^*(\overline{q_1}w;q_2),
\end{equation}
where
\begin{equation}\label{eqn: H^*}
  H^*(w;q) = \sum_{\substack{u,v(\mod q)\\(uv-1,q)=1}} \chi_q(uv(u+1)(v+1))e_q((uv-1)w),
\end{equation}
and from~\cite[Eq. (11.9)]{conrey2000cubic}, we have
\begin{equation}\label{eqn: H^*=}
  H^*(w;q) = \frac{1}{\varphi(q)}\sum_{\psi(\mod q)} \tau(\bar{\psi})g(\chi,\psi)\psi(w),
\end{equation}
where $\tau(\psi)$ is the Gauss sum, and $g(\chi,\psi)\ll q^{1+\ve}$.

We first recall the following hybrid large sieve.

\begin{lemma}\label{lemma: HLS}
  Suppose $U\geq1$, and let $a_n$ be a sequence of complex numbers. Then
  \begin{equation}\label{eqn: HLS}
    \int_{-U}^{U} \sum_{\psi(\mod q)} \bigg|\sum_{n\leq N}a_n\psi(n)n^{it}\bigg|^2dt \ll (qU+N)\sum_{n\leq N}|a_n|^2.
  \end{equation}
\end{lemma}
\begin{proof}
  See Gallagher~\cite[Theorem 2]{gallagher1970large}.
\end{proof}

\begin{lemma}\label{lemma: average H}
  Suppose that $q$ is squarefree. Let $a_1,a_2,b_1,b_2\in\bbZ$, $c\in\bbZ$, such that $(b_1b_2c,q)=(a_1a_2,c)=1$.
  Let $N_2,D_2\geq1$. Suppose $\alpha_d,\beta_n\in\bbC$ with $|\alpha_d|\leq1$ for $1\leq d\leq D_2$,
  $1\leq n\leq N_2$, and $U\geq1$.
  Then for any $\ve>0$, we have
  \begin{equation}\label{eqn: average H}
    \begin{split}
         & \int_{|t|\ll U} \bigg|\sum_{\substack{n\asymp N_2,d\asymp D_2\\(nd,qc)=1}} \alpha_d\beta_ne(a_1\overline{a_2}\bar{d}n/c)H(b_1\overline{b_2}\bar{d}n;q)\left(\frac{n}{d}\right)^{it}\bigg|dt \\
         & \hskip 150pt \ll \frac{q^{1/2+\ve}}{c^{1/2-\ve}} (qcU+D_2)^{1/2} D_2^{1/2} (qcU+N_2)^{1/2}\|\beta\|,
    \end{split}
  \end{equation}
  where as usual $\|\beta\|=\left(\sum|\beta_n|^2\right)^{1/2}$.
\end{lemma}

This is a variation of~\cite[Lemma 11.1]{conrey2000cubic}. We combine the ingredients in both~\cite[Lemma 13]{blomer2012subconvexity} and~\cite[Lemma 9.2]{young2014weyl}.

\begin{proof}
  By \eqref{eqn: H to H^*}, we have
  \[
    \begin{split}
         & \int_{|t|\ll U} \bigg|\sum_{\substack{n\asymp N_2,d\asymp D_2\\(nd,qc)=1}} \alpha_d\beta_ne(a_1\overline{a_2}\bar{d}n/c)H(b_1\overline{b_2}\bar{d}n;q)\left(\frac{n}{d}\right)^{it}\bigg|dt \\
         & \hskip 60pt \leq \sum_{q_1q_2=q}\int_{|t|\ll U} \bigg|\sum_{\substack{n\asymp N_2,d\asymp D_2\\(nd,qc)=1}} \alpha_d\beta_ne(a_1\overline{a_2}\bar{d}n/c)H^*(b_1\overline{b_2}\bar{d}n\overline{q_1};q_2)\left(\frac{n}{d}\right)^{it}\bigg|dt.
    \end{split}
  \]
  We just handle the case $q_2=q$, since the other cases turn out to have a smaller upper bound.
  By \eqref{eqn: H^*=}, we have
  \[
    \begin{split}
         & \int_{|t|\ll U} \bigg|\sum_{\substack{n\asymp N_2,d\asymp D_2\\(nd,qc)=1}} \alpha_d\beta_ne(a_1\overline{a_2}\bar{d}n/c)H^*(b_1\overline{b_2}\bar{d}n;q)\left(\frac{n}{d}\right)^{it}\bigg|dt \\
         & \hskip 25pt \ll \int_{|t|\ll U} \frac{1}{\varphi(qc)}\sum_{\psi(q)}\sum_{\omega(c)}
         |\tau(\bar{\psi})\tau(\bar{\omega})g(\chi,\psi)|
         \bigg|\sum_{\substack{n\asymp N_2,d\asymp D_2\\(nd,qc)=1}} \alpha_d\beta_n\psi\omega(\bar{d}n)\left(\frac{n}{d}\right)^{it}\bigg|dt \\
         & \hskip 25pt \ll \frac{q^{1/2+\ve}}{c^{1/2-\ve}}
            \bigg(\int_{|t|\ll U}\sum_{\psi\omega(qc)}\bigg|\sum_d \alpha_d \psi\omega(\bar{d})d^{-it}\bigg|^2 dt \bigg)^{1/2}
            \bigg(\int_{|t|\ll U}\sum_{\psi\omega(qc)}\bigg|\sum_n \beta_n \psi\omega(n)n^{it}\bigg|^2 dt \bigg)^{1/2},
    \end{split}
  \]
  where $|\alpha_d|\leq 1$.
  Now after applying Lemma \ref{lemma: HLS}, we prove the lemma.
\end{proof}

\section{Proof of Theorem \ref{thm: q}}\label{sec: thm q}

Denote $R_q(n)=S(n,0;q)$ the Ramanujan sum.
By a long and complicated computation, Blomer~\cite[Eq. (51)]{blomer2012subconvexity} gave
\begin{equation}\label{eqn: cS=}
  \begin{split}
    & \cS_\sigma(q,N;\delta)
       = \gamma \sum_{\pm} \sum_{\delta_0\delta'=\delta}\frac{\mu(\delta_0)\chi(\delta)}{\delta_0}
                  \sum_{c_1'c_2'=q}
                  \sum_{\substack{f_1,f_2,d_2'\\ (f_1f_2d_2',c_2'\delta)=1\\ (f_1,f_2)=1,\mu^2(f_1)=1\\(f_1f_2,qd_2')=1}}
                  \sum_{\substack{n_1'|f_1c_1'\\(n_1',d_2')=1}}
                  \sum_{\substack{n_2\\ (n_2,d_2')=1}}\frac{d_2'\mu(f_2)}{c_2'} \\
       & \hskip 10pt \cdot \frac{A(n_2,n_1'f_2)}{n_1'n_2} \frac{\varphi(f_1f_2d_2'c_1')\varphi(f_1d_2'c_1'/n_1')}{\varphi(f_1f_2d_2'c_1'c_2')^2}
                e\left(\pm\sigma\frac{(n_1'c_2')^2f_2n_2\delta_0\overline{d_2'c_1'}}{f_1\delta'}\right)
       \frac{h\chi_h(-1)}{\varphi(k)} R_k(n_2n_1'f_2c_2')  \\
       & \hskip 20pt \cdot R_k(c_2'\delta_0)R_k(n_1'f_2c_2') H(\mp\sigma\overline{f_1d_2'hk}n_2(n_1'c_2')^2f_2c_2'\delta_0\overline{\delta'},\ell)
                \check{\Psi}_\sigma^\pm\left(\frac{n_2(n_1')^2N}{(f_1d_2'c_1')^3f_2},\frac{\sqrt{\delta N}}{qf_1f_2d_2'\delta_0}\right),
  \end{split}
\end{equation}
where $\varphi$ is the Euler function, $\gamma=\pi^{3/2}\chi(-1)/2$, and
\begin{equation}\label{eqn: h,k,l}
  h=(f_1f_2d_2',q)=(d_2',q),\quad k=(n_2(n_1'f_2c_2')^2c_2,q/h), \quad \ell=q/(hk).
\end{equation}
We summarize the relations of these variables and previous variables here,
although we don't need them in this section
(see Blomer~\cite[\S6]{blomer2012subconvexity})
\begin{gather}
       c=qr, \quad c_2=c/c_1, \quad \delta_0=(\delta,r), \quad \delta'=\delta/\delta_0,\nonumber\\
       c'=c/\delta_0, \quad c_2'=c_2/\delta_0,\quad r'=r/\delta_0, \label{eqn: relations} \\
       n_1'=n_1/f_2, \quad c_1'=c_1/r', \quad f_1f_2d_2'=r'.\nonumber
\end{gather}
As Blomer~\cite[\S7]{blomer2012subconvexity} did, in the $q$-aspect,
one can use the decay conditions
of $\check{\Psi}_\sigma^\pm$ to show that several variables can be dropped.
But in our case, things become much more complicated. However,
the argument is similar to Blomer~\cite[\S7]{blomer2012subconvexity}, and we need to
track the dependence on $T$ and $M$. One can
see that the argument in \S \ref{subsec: c_2'=q s=-1}--\ref{subsec: k=q s=-1}
is similar to \S \ref{subsec: main case s=-1}, and even easier.
In the $q$-aspect, results in \S \ref{subsec: c_2'=q s=-1}--\ref{subsec: k=q s=-1}
are better. However, it seems that to get a good bound in the $t$-aspect,
we have to use the large sieve in all cases.

\subsection{The main case}\label{subsec: main case s=-1}

We first deal with the main case, that is, $c_1'=q,\ c_2'=h=k=1$.
This is the most important case (at least in the $q$-aspect),
so we give the details of the treatment of this case.
Denote these terms in \eqref{eqn: cS=} as $\cS_\sigma^\dag(q,N;\delta)$.
Note that we have $(d_2'n_1'n_2,q)=1$. Write $f_1=n_1'g$.  Then we have
\begin{equation}\label{eqn: S_s^d}
  \begin{split}
       & \cS_\sigma^\dag(q,N;\delta)
       = \gamma \sum_{\pm} \sum_{\delta_0\delta'=\delta} \frac{\mu(\delta_0)\chi(\delta)}{\delta_0}
                  \sum_{\substack{g,n_1',f_2,d_2'\\(gn_1'f_2d_2',q\delta)=1\\(gn_1',f_2)=1,\mu^2(gn_1')=1\\(gn_1'f_2,d_2')=1}}
                  \frac{d_2'\mu(f_2)}{\varphi(n_1'f_2)n_1'}
        \sum_{\substack{n_2\\(n_2,qd_2')=1}}\frac{A(n_2,n_1'f_2)}{n_2} \\
       & \hskip 30pt \cdot e\left(\pm\sigma\frac{n_1'f_2n_2\delta_0\overline{d_2'q}}{g\delta'}\right)
            H(\mp\sigma\overline{gd_2'}n_2n_1'f_2\delta_0\overline{\delta'},q)
                \check{\Psi}_\sigma^\pm\left(\frac{n_2N}{(gd_2'q)^3n_1'f_2},
                \frac{\sqrt{\delta N}}{qgn_1'f_2d_2'\delta_0}\right).
  \end{split}
\end{equation}
Since we have $(n_1'f_2\delta_0\overline{d_2'q},g\delta')=1$, let
\[
  s=(n_2,g\delta'), \quad n_2=n_2's,\quad  (n_2',g\delta'/s)=1.
\]
We cancel the factor $s$ from the numerator and denominator of the exponential
getting
\begin{equation}\label{eqn: cS<<}
  \begin{split}
      \cS_\sigma^\dag(q,N;\delta) & = \gamma \sum_{\pm} \sum_{\delta_0\delta'=\delta} \frac{\mu(\delta_0)\chi(\delta)}{\delta_0}
                  \sum_{\substack{g,n_1',f_2\\(gn_1'f_2,q\delta)=1\\(gn_1',f_2)=1,\mu^2(gn_1')=1}}
                  \frac{\mu(f_2)}{\varphi(n_1'f_2)n_1'}
                  \sum_{s|g\delta'} \frac{1}{s}  \\
       & \hskip 30pt \cdot \sum_{\substack{d_2'\\(d_2',q\delta gn_1'f_2)=1}}
                \sum_{\substack{n_2'\\(n_2',qd_2'g\delta'/s)=1}}\frac{d_2'A(n_2's,n_1'f_2)}{n_2'}
                e\left(\pm\sigma\frac{n_1'f_2n_2'\delta_0\overline{d_2'q}}{g\delta'/s}\right) \\
       & \hskip 80pt \cdot  H(\mp\sigma\overline{gd_2'}n_2'sn_1'f_2\delta_0\overline{\delta'},q)
                \check{\Psi}_\sigma^\pm\left(\frac{n_2'sN}{(gd_2'q)^3n_1'f_2},
                \frac{\sqrt{\delta N}}{qgn_1'f_2d_2'\delta_0}\right).
  \end{split}
\end{equation}

The main factors in \eqref{eqn: cS<<} are the variables $d_2'$ and $n_2'$.
We open the coprimality condition $(d_2',n_2')=1$ by M\"{o}bius inversion.
We introduce a new variable $r$ and get
\[
  \begin{split}
      \cS_\sigma^\dag(q,N;\delta) & \ll (qT)^\ve \sup_{N_2,D_2} \sum_{\pm} \sum_{\delta_0\delta'=\delta} \frac{1}{\delta_0}
                  \sum_{\substack{g,n_1',f_2,r\\(gn_1'f_2,q\delta)=1\\(gn_1',f_2)=1,\mu^2(gn_1')=1\\(r,q\delta gn_1'f_2)=1}}
                  \sum_{s|g\delta'} \frac{1}{f_2(n_1')^2s} \\
       & \hskip 10pt \cdot  \bigg| \sum_{\substack{N_2\leq n_2'\leq 2N_2\\(n_2',qg\delta'/s)=1}}
                \sum_{\substack{D_2\leq d_2'\leq 2D_2 \\ (d_2',q\delta gn_1'f_2)=1}}
                \frac{d_2'A(n_2'rs,n_1'f_2)}{n_2'}
                e\left(\pm\sigma\frac{n_1'f_2n_2'\delta_0\overline{d_2'q}}{g\delta'/s}\right) \\
       & \hskip 50pt \cdot   H(\mp\sigma\overline{gd_2'}n_2'sn_1'f_2\delta_0\overline{\delta'},q)
                \check{\Psi}_\sigma^\pm\left(\frac{n_2'sN}{(gd_2'q)^3r^2n_1'f_2},
                \frac{\sqrt{\delta N}}{qgn_1'f_2d_2'r\delta_0}\right)\bigg|.
  \end{split}
\]
By Lemma \ref{lemma: Psi}, we can assume that
\begin{equation}\label{eqn: D_2}
  \left\{\begin{array}{ll}
    1\leq D_2 \leq \frac{\sqrt{\delta N}}{qgn_1'f_2r\delta_0 TM^{1-\ve}}, & \textrm{if } \sigma=1, \\
    1\leq D_2 \asymp \frac{\sqrt{\delta N}}{qgn_1'f_2r\delta_0 T}, & \textrm{if } \sigma=-1,
  \end{array}\right.
\end{equation}
and if $x=\frac{n_2'sN}{(gd_2'q)^3r^2n_1'f_2}\gg T^\ve$, then
\begin{equation}\label{eqn: N_2}
  \left\{\begin{array}{ll}
    1\leq N_2 \asymp \frac{(\delta^3N)^{1/2}}{(n_1'f_2)^2 r\delta_0^3 s}, & \textrm{if } \sigma=1, \\
    1\leq N_2 \leq \frac{(\delta^3N)^{1/2}}{(n_1'f_2)^2 r\delta_0^3 s M^{3-\ve}}, & \textrm{if } \sigma=-1.
  \end{array}\right.
\end{equation}

Now we consider the case $\sigma=-1$, and $x\gg T^\ve$.
By Lemma \ref{lemma: Psi}, we infer
\[
  \begin{split}
       \cS_\sigma^\dag(q,N;\delta) &
            \ll (qT)^\ve  M \sum_{\pm} \sum_{\delta_0\delta'=\delta}\frac{1}{\delta_0}
                  \sum_{\substack{g,n_1',f_2,r\\(gn_1'f_2,q\delta)=1\\(gn_1',f_2)=1,\mu^2(gn_1')=1\\(r,q\delta gn_1'f_2)=1}}
                  \frac{1}{(gf_2)^{3/2}(n_1')^{5/2}r}
                  \sum_{s|g\delta'} \frac{1}{s^{1/2}}\\
       & \hskip 20pt \cdot  \sup_{D_2,N_2} \frac{N^{1/2}}{q^{3/2}N_2^{1/2}D_2^{1/2}}
                \int_{|t|\ll U} \bigg| \sum_{\substack{N_2\leq n_2'\leq 2N_2\\(n_2',qg\delta'/s)=1}}
                \sum_{\substack{D_2\leq d_2'\leq 2D_2 \\ (d_2',q\delta gn_1'f_2)=1}}
                \alpha(d_2')\beta(n_2')  \\
       & \hskip 20pt \cdot  A(n_2'rs,n_1'f_2) e\left(\pm\sigma\frac{n_1'f_2n_2'\delta_0\overline{d_2'q}}{g\delta'/s}\right)
                 H(\mp\sigma\overline{gd_2'}n_2'sn_1'f_2\delta_0\overline{\delta'},q)
                 \left(\frac{n_2'}{d_2'}\right)^{it}\bigg|dt,
  \end{split}
\]
for certain coefficients $\alpha(d_2'),\beta(n_2')$ of absolute value at most $1$, and $D_2,N_2$ restricted as in \eqref{eqn: D_2} and \eqref{eqn: N_2}.
Note that by \eqref{eqn: X&D-} we have $U\asymp x^{1/3}\ll T^{1+\ve}/M$,
and from \eqref{eqn: blomer} we have
\[
  \bigg(\sum_{n_2'\asymp N_2}|A(n_2'rs,n_1'f_2)|^2\bigg)^{1/2} \ll (qT)^\ve N_2^{1/2}(rsn_1'f_2)^{7/32}.
\]
Applying Lemma \ref{lemma: average H}, we obtain
\[
  \begin{split}
     \frac{1}{N^{1/2}}\cS_\sigma^\dag(q,N;\delta)
       & \ll (qT)^\ve  \frac{M}{q} \sum_{\delta_0\delta'=\delta}\frac{1}{\delta_0}
                  \sum_{g,n_1',f_2,r}   \frac{1}{(gf_2)^{3/2}(n_1')^{5/2}r}
                  \sum_{s|g\delta'} \frac{1}{s^{1/2}}\\
       & \hskip 60pt \cdot  \sup_{D_2,N_2} (qU+D_2)^{1/2} (qUg\delta'/s+N_2)^{1/2} (rsn_1'f_2)^{7/32} \\
       & \ll (qT)^\ve  \frac{M}{q} \sum_{\delta_0\delta'=\delta}\frac{1}{\delta_0}
                  \sum_{g,n_1',f_2,r}   \frac{1}{(gf_2)^{3/2}(n_1)^{5/2}r}
                  \sum_{s|g\delta'} \frac{1}{s^{1/2}} \\
       & \hskip 60pt \cdot  \left(\frac{qT}{M}\left(\frac{g\delta}{s}\right)^{1/2}
                +\left(\frac{qT}{M}\right)^{1/2}\left(\frac{(qT)^{3/2}}{rsM^{3}}\right)^{1/2}\right)
                (rsn_1'f_2)^{7/32}\\
       & \ll (qT)^\ve \left( T\delta^{1/2}(qT)^{7/64} + q^{1/4}T^{5/4}M^{-1}\right)
         \ll (qT)^\ve \delta^{1/2} q^{1/4}TM,
  \end{split}
\]
provided $T^{1/8+\ve}\ll M \ll T^{1/2}$. Here we use the fact
$$rg\delta n_1'f_2\ll (qT)^{1/2+\ve},$$
which is a consequence of \eqref{eqn: D_2} and \eqref{eqn: N}.
Note that if $\sigma=1$, and $x\gg T^\ve$, then the same argument will give
\[
  \begin{split}
     \frac{1}{N^{1/2}}\cS_\sigma^\dag(q,N;\delta)
       & \ll (qT)^\ve  \frac{M}{q} \sum_{\delta_0\delta'=\delta}\frac{1}{\delta_0}
                  \sum_{g,n_1',f_2,r}   \frac{1}{(gf_2)^{3/2}(n_1')^{5/2}r}
                  \sum_{s|g\delta'} \frac{1}{s^{1/2}} \\
       & \hskip 60pt \cdot  \left(\frac{qT}{M}\left(\frac{g\delta}{s}\right)^{1/2}
                +\left(\frac{qT}{M}\right)^{1/2}\left(\frac{(qT)^{3/2}}{rs}\right)^{1/2}\right)
                (rsn_1'f_2)^{7/32}\\
       & \ll (qT)^\ve \left( T\delta^{1/2}(qT)^{7/64} + q^{1/4}T^{5/4}M^{1/2}\right)
        \ll (qT)^\ve \delta^{1/2} q^{1/4}TM,
  \end{split}
\]
provided $M\asymp T^{1/2}$.
Note that this will not give us a subconvexity bound in the $t$-aspect.

Now we consider the case $x\ll T^\ve$.
Since $D=\frac{\sqrt{\delta N}}{qgn_1'f_2d_2'r\delta_0}\gg T$,
we have $x/D^3\ll T^{\ve-3}$. Hence
\begin{equation}\label{eqn:N2forx<<}
  N_2 \ll \frac{(\delta^3N)^{1/2}T^\ve}{(n_1'f_2)^2 r\delta_0^3 s T^{3}}
  \ll (qT)^\ve \frac{q^{3/2}}{rsT^{3/2}}.
\end{equation}
Also note that the upper bound of $N$ implies that this will happen only if $q\gg T^{1-\ve}$.
By Lemma \ref{lemma: Psi}, for both $\sigma=\pm1$, we have
\[
  \begin{split}
        \cS_\sigma^\dag(q,N;\delta) &
            \ll (qT)^\ve  T \sum_{\pm} \sum_{\delta_0\delta'=\delta}\frac{1}{\delta_0}
                  \sum_{\substack{g,n_1',f_2,r\\(gn_1'f_2,q\delta)=1\\(gn_1',f_2)=1,\mu^2(gn_1')=1\\(r,q\delta gn_1'f_2)=1}}
                  \frac{1}{(gf_2)^{3/2}(n_1')^{5/2}r}
                  \sum_{s|g\delta'} \frac{1}{s^{1/2}} \\
       & \hskip 30pt \cdot  \sup_{D_2,N_2} \frac{N^{1/2}}{q^{3/2}N_2^{1/2}D_2^{1/2}}
                \bigg| \sum_{\substack{N_2\leq n_2'\leq 2N_2\\(n_2',qg\delta'/s)=1}}
                \sum_{\substack{D_2\leq d_2'\leq 2D_2 \\ (d_2',q\delta gn_1'f_2)=1}}
                \alpha(d_2')\beta(n_2')  \\
       & \hskip 60pt \cdot  A(n_2'rs,n_1'f_2) e\left(\pm\sigma\frac{n_1'f_2n_2'\delta_0\overline{d_2'q}}{g\delta'/s}\right)
                 H(\mp\sigma\overline{gd_2'}n_2'sn_1'f_2\delta_0\overline{\delta'},q) \bigg|,
  \end{split}
\]
for certain coefficients $\alpha(d_2'),\beta(n_2')$ of absolute value at most $1$, and $D_2,N_2$ restricted as in \eqref{eqn: D_2} and \eqref{eqn:N2forx<<}.
Now by Blomer~\cite[Lemma 13]{blomer2012subconvexity}, we have
\[
  \begin{split}
     \frac{1}{N^{1/2}}\cS_\sigma^\dag(q,N;\delta)
       & \ll (qT)^\ve  \frac{T}{q} \sum_{\delta_0\delta'=\delta}\frac{1}{\delta_0}
                  \sum_{g,n_1',f_2,r}   \frac{1}{(gf_2)^{3/2}(n_1')^{5/2}r}
                  \sum_{s|g\delta'} \frac{1}{s^{1/2}}\\
       & \hskip 90pt \cdot  \sup_{D_2,N_2} (q+D_2)^{1/2} (qg\delta'/s+N_2)^{1/2} (rsn_1'f_2)^{7/32} \\
       & \ll (qT)^\ve  \frac{T}{q} \sum_{\delta_0\delta'=\delta}\frac{1}{\delta_0}
                  \sum_{g,n_1',f_2,r}   \frac{1}{(gf_2)^{3/2}(n_1')^{5/2}r}
                  \sum_{s|g\delta'} \frac{1}{s^{1/2}} \\
       & \hskip 90pt \cdot  \left(q\left(\frac{g\delta}{s}\right)^{1/2}
                + q^{1/2}\left(\frac{q^{3/2}}{rsT^{3/2}}\right)^{1/2}\right)
                (rsn_1'f_2)^{7/32}\\
       & \ll (qT)^\ve \left( T\delta^{1/2}(qT)^{7/64} + q^{1/4}T^{1/4}\right)
         \ll (qT)^\ve \delta^{1/2} q^{1/4}TM,
  \end{split}
\]
provided $T^{1/8+\ve}\ll M \ll T^{1/2}$.
Hence in both cases, we prove \eqref{eqn: sum over cS} under the assumption $M\asymp T^{1/2}$.


\subsection{The case $c_2'=q$}\label{subsec: c_2'=q s=-1}

In this subsection, we prove the terms attached with
$c_2'=q$, $c_1'=1$ have a good bound, that is, we show that we can assume $c_2'=1$.
Denote these terms in \eqref{eqn: cS=} as $\cS_\sigma^\flat(q,N;\delta)$.
Let $D=\frac{\sqrt{\delta N}}{qf_1f_2d_2'\delta_0}$. By Lemma \ref{lemma: Psi}, we can assume
\[
  \left\{
  \begin{array}{ll}
    D\gg TM^{1-\ve}, & \textrm{if } \sigma=1, \\
    D\asymp T, & \textrm{if } \sigma=-1.
  \end{array}\right.
\]
Hence, by \eqref{eqn: N} and \eqref{eqn: relations}, we have
\[
  x=\frac{n_2(n_1')^2N}{(f_1d_2'c_1')^3f_2}
   =\frac{n_2(n_1')^2 f_2^2\delta_0^3 (qD)^3}{(\delta^3 N)^{1/2}} \gg T^\ve
\]
By the condition $(f_1f_2d_2',q\delta)=1$, we have $h=(d_2',q)=1$, and then $k=q$ and $\ell=1$.
Hence,
by the same process we obtain
\[
  \begin{split}
      \cS_\sigma^\flat(q,N;\delta)
       & = \gamma \sum_{\pm} \sum_{\delta_0\delta'=\delta}\frac{\mu(\delta_0)\chi(\delta)}{\delta_0}
                  \sum_{\substack{g,n_1',f_2,r\\ (gn_1'f_2,q\delta)=1\\ (gn_1',f_2)=1,\mu^2(gn_1')=1\\(r,q\delta gn_1'f_2)=1}}
                  \frac{\mu(f_2)\mu(r)}{n_1'\varphi(n_1'f_2)}
                  \sum_{s|g\delta'}\frac{1}{s}
                   \sum_{\substack{d_2'\\(d_2',q\delta gn_1'f_2)=1}} d_2' \\
       & \hskip 5pt \cdot \sum_{\substack{n_2'\\(n_2',g\delta'/s)=1}}
                 \frac{A(n_2'rs,n_1'f_2)}{qn_2'}  e\left(\pm\sigma\frac{n_1'q^2f_2n_2'\delta_0\overline{d_2'}}{g\delta'/s}\right)
                \check{\Psi}_\sigma^\pm\left(\frac{n_2's N}{(gd_2')^3r^2n_1'f_2},\frac{\sqrt{\delta N}}{qgn_1'f_2d_2'r\delta_0}\right).
  \end{split}
\]
Then by Lemma \ref{lemma: Psi}, we have
\[
  \begin{split}
     &  \frac{1}{N^{1/2}}\cS_\sigma^\flat(q,N;\delta)
       \ll (qT)^\ve \sum_{\pm} \sum_{\delta_0\delta'=\delta}\frac{1}{\delta_0}
                  \sum_{g,n_1',f_2,r}
                  \frac{1}{(gf_2)^{3/2}(n_1')^{5/2}r} \sum_{s|g\delta'}\frac{1}{s^{1/2}}
       \sup_{D_2,N_2} \frac{M}{qD_2^{1/2}N_2^{1/2}} \\
       & \hskip 5pt \cdot \int_{|t|\ll U}
                \bigg|\sum_{\substack{d_2'\asymp D_2\\(d_2',q\delta gn_1'f_2)=1}}
                \sum_{\substack{n_2'\asymp N_2\\(n_2',g\delta'/s)=1}}  \alpha(d_2')\beta(n_2') A(n_2'rs,n_1'f_2)
                e\left(\pm\sigma\frac{n_1'q^2f_2n_2'\delta_0\overline{d_2'}}{g\delta'/s}\right)
                \left(\frac{n_2'}{d_2'}\right)^{it}\bigg|dt,
  \end{split}
\]
for certain coefficients $\alpha(d_2'),\beta(n_2')$ of absolute value at most $1$.
Here we can assume that
\begin{equation*}
  \left\{\begin{array}{ll}
    1\leq D_2 \leq \frac{\sqrt{\delta N}}{qgn_1'f_2r\delta_0 TM^{1-\ve}}, & \textrm{if } \sigma=1, \\
    1\leq D_2 \asymp \frac{\sqrt{\delta N}}{qgn_1'f_2r\delta_0 T}, & \textrm{if } \sigma=-1,
  \end{array}\right.
\end{equation*}
and
\begin{equation*}
  \left\{\begin{array}{ll}
    1\leq N_2 \asymp \frac{(\delta^3N)^{1/2}}{q^3(n_1'f_2)^2 r\delta_0^3 s}, & \textrm{if } \sigma=1, \\
    1\leq N_2 \leq \frac{(\delta^3N)^{1/2}}{q^3(n_1'f_2)^2 r\delta_0^3 s M^{3-\ve}}, & \textrm{if } \sigma=-1.
  \end{array}\right.
\end{equation*}
Note that in both cases, we have $U\ll T^{1+\ve}/M$.
We can use the multiplicative characters to separate the variables in the exponential function,
together with Cauchy--Schwarz inequality and Lemma \ref{lemma: HLS}, we obtain
\[
  \begin{split}
     \frac{1}{N^{1/2}}\cS_\sigma^\flat(q,N;\delta)
       & \ll (qT)^\ve \sum_{\pm} \sum_{\delta_0\delta'=\delta}\frac{1}{\delta_0}
                  \sum_{g,n_1',f_2,r} \frac{1}{(gf_2)^{3/2}(n_1')^{5/2}r} \sum_{s|g\delta'}\frac{1}{s^{1/2}}  \\
       & \hskip 60pt \cdot \sup_{D_2,N_2}  \frac{M}{q} (g\delta'/s)^{1/2}
                (U+D_2)^{1/2}(U+N_2)^{1/2} (rsn_1'f_2)^{7/32}.
  \end{split}
\]
In the case $\sigma=-1$, we have
\[
  \begin{split}
     \frac{1}{N^{1/2}}\cS_\sigma^\flat(q,N;\delta)
       & \ll (qT)^\ve \sum_{\pm} \sum_{\delta_0\delta'=\delta}\frac{1}{\delta_0}
                  \sum_{g,n_1',f_2,r} \frac{1}{(gf_2)^{3/2}(n_1')^{5/2}r}
                  \sum_{s|g\delta'}\frac{1}{s^{1/2}}  \\
       & \hskip 100pt \cdot \frac{M}{q^{3/4}} (g\delta'/s)^{1/2}
                \left(\frac{T}{M} + \frac{T^{5/4}}{M^2r^{1/2}} \right)
                (rsn_1'f_2)^{7/32} \\
       & \ll (qT)^\ve \delta^{1/2} TM,
  \end{split}
\]
provided $T^{1/8+\ve}\ll M\ll T^{1/2}$,
which is good enough for our purpose by \eqref{eqn: sum over cS}.
In the case $\sigma=1$, the same argument will give
\[
  \frac{1}{N^{1/2}}\cS_\sigma^\flat(q,N;\delta)   \ll (qT)^\ve \delta^{1/2} TM,
\]
provided $M\asymp T^{1/2}$.
From now on we assume $c_2'=1, c_1'=q$.

\subsection{The case $h=q$}\label{subsec: h=q s=-1}

Next we show that the contribution from the terms with $h=q$ is negligible.
Denote these terms in \eqref{eqn: cS=} as
$\cS_\sigma^\natural(q,N;\delta)$.
In this case, we have $k=\ell=1$ and $q|d_2'$, and hence we also get $n_1'|f_1$.
By a similar argument (writing $d_2'=qd_2''$), we get
\begin{equation*}
  \begin{split}
      \cS_\sigma^\natural(q,N;\delta)
       & = \gamma \sum_{\pm} \sum_{\delta_0\delta'=\delta}\frac{\mu(\delta_0)\chi(\delta)}{\delta_0}
                  \sum_{\substack{g,n_1',f_2,r\\ (gn_1'f_2,q\delta)=1\\ (gn_1',f_2)=1,\mu^2(gn_1')=1\\(r,q\delta gn_1'f_2)=1}}
                  \frac{\mu(f_2)\mu(r)}{n_1'\varphi(n_1'f_2)}
                  \sum_{s|g\delta'} \frac{1}{s}
                   \\
       & \hskip 30pt \cdot q^2\chi_q(-1) \sum_{\substack{d_2''\\ (d_2'',\delta gn_1'f_2)=1}}
                  \sum_{\substack{n_2'\\ (n_2',qg\delta'/s)=1}} \frac{d_2''A(n_2'rs,n_1'f_2)}{n_2'}
                 \\
       & \hskip 60pt \cdot e\left(\pm\sigma\frac{n_1'f_2n_2'\delta_0\overline{d_2''q^2}}{g\delta'/s}\right)
                \check{\Psi}_\sigma^\pm\left(\frac{n_2'sN}{(gd_2''q^2)^3r^2n_1'f_2},
                \frac{\sqrt{\delta N}}{q^2gn_1'f_2d_2''r\delta_0}\right).
  \end{split}
\end{equation*}
Now we consider the sum over $d_2''\asymp D_2, n_2'\asymp N_2$, and we can assume
\begin{equation*}
  \left\{\begin{array}{ll}
    1\leq D_2 \leq \frac{\sqrt{\delta N}}{q^2gn_1'f_2r\delta_0 TM^{1-\ve}}, & \textrm{if } \sigma=1, \\
    1\leq D_2 \asymp \frac{\sqrt{\delta N}}{q^2gn_1'f_2r\delta_0 T}, & \textrm{if } \sigma=-1,
  \end{array}\right.
\end{equation*}
and if $x=\frac{n_2'sN}{(gd_2''q^2)^3r^2n_1'f_2}\gg T^\ve$, then
\begin{equation*}
  \left\{\begin{array}{ll}
    1\leq N_2 \asymp \frac{(\delta^3N)^{1/2}}{(n_1'f_2)^2 r\delta_0^3 s}, & \textrm{if } \sigma=1, \\
    1\leq N_2 \leq \frac{(\delta^3N)^{1/2}}{(n_1'f_2)^2 r\delta_0^3 s M^{3-\ve}}, & \textrm{if } \sigma=-1.
  \end{array}\right.
\end{equation*}
If $x\gg T^\ve$ and $\sigma=-1$, then by Lemma \ref{lemma: Psi}, we have
\begin{equation*}
  \begin{split}
      &  \frac{1}{N^{1/2}} \cS_\sigma^\natural(q,N;\delta)
         \ll  \sum_{\pm} \sum_{\delta_0\delta'=\delta}\frac{1}{\delta_0}
                  \sum_{g,n_1',f_2,r} \frac{1}{(gf_2)^{3/2}(n_1')^{5/2}r}
                  \sum_{s|g\delta'} \frac{1}{s^{1/2}}
          \sup_{D_2,N_2} \frac{M}{q(D_2N_2)^{1/2}} \\
       & \hskip 15pt \cdot \int_{|t|\ll U}
            \bigg| \sum_{\substack{d_2''\asymp D_2\\ (d_2'',\delta gn_1'f_2)=1}}
                  \sum_{\substack{n_2'\asymp N_2 \\ (n_2',qg\delta'/s)=1}} \alpha(d_2'')\beta(n_2')
       A(n_2'rs,n_1'f_2) e\left(\pm\sigma\frac{n_1'f_2n_2'\delta_0\overline{d_2''q^2}}{g\delta'/s}\right)
                \left(\frac{n_2'}{d_2''}\right)^{it} \bigg|dt,
  \end{split}
\end{equation*}
for certain coefficients $\alpha(d_2''),\beta(n_2')$ of absolute value at most $1$.
Hence by Lemma \ref{lemma: HLS} again, we have
\begin{equation*}
  \begin{split}
     \frac{1}{N^{1/2}} \cS_\sigma^\natural(q,N;\delta)
       & \ll  \sup_{D_2,N_2} (qT)^\ve \frac{M}{q}r^{7/32}(U+D_2)^{1/2}(\delta U+N_2)^{1/2}\\
       & \ll  (qT)^\ve \delta^{1/2}\left(\frac{T}{q}r^{7/32}+\frac{T^{5/4}}{q^{1/4}M}\right) \ll (qT)^\ve TM\delta^{1/2},
  \end{split}
\end{equation*}
provided $T^{1/8+\ve}\ll M\ll T^{1/2}$.
And again, if $x\gg T^\ve$ and $\sigma=1$, the same argument shows that
\begin{equation*}
  \begin{split}
     \frac{1}{N^{1/2}} \cS_\sigma^\natural(q,N;\delta)
       & \ll (qT)^\ve \delta^{1/2}\left(\frac{T}{q}r^{7/32}+\frac{T^{5/4}M^{1/2}}{q^{1/4}}\right) \ll (qT)^\ve TM\delta^{1/2},
  \end{split}
\end{equation*}
provided $M\asymp T^{1/2}$.

Now if $x\ll T^\ve$, then by Lemma \ref{lemma: Psi}, for both $\sigma=\pm1$, we have
\begin{equation*}
  \begin{split}
     \frac{1}{N^{1/2}} \cS_\sigma^\natural(q,N;\delta)
       & \ll  \sum_{\pm} \sum_{\delta_0\delta'=\delta}\frac{1}{\delta_0}
                  \sum_{g,n_1',f_2,r} \frac{1}{(gf_2)^{3/2}(n_1')^{5/2}r}
                  \sum_{s|g\delta'} \frac{1}{s^{1/2}}
       \sup_{D_2,N_2} \frac{T}{q(D_2N_2)^{1/2}} \\
       & \hskip 10pt \cdot   \bigg| \sum_{\substack{d_2''\asymp D_2\\ (d_2'',\delta gn_1'f_2)=1}}
                  \sum_{\substack{n_2'\asymp N_2 \\ (n_2',qg\delta'/s)=1}} \alpha(d_2'')\beta(n_2')
       A(n_2'rs,n_1'f_2) e\left(\pm\sigma\frac{n_1'f_2n_2'\delta_0\overline{d_2''q^2}}{g\delta'/s}\right)\bigg|.
  \end{split}
\end{equation*}
Hence by Blomer~\cite[Lemma 13]{blomer2012subconvexity}, we have
\begin{equation*}
  \begin{split}
     \frac{1}{N^{1/2}} \cS_\sigma^\natural(q,N;\delta)
       & \ll  \sup_{D_2,N_2} (qT)^\ve \frac{T}{q}r^{7/32}(1+D_2)^{1/2}(\delta+N_2)^{1/2}\\
       & \ll  (qT)^\ve \delta^{1/2}\left(\frac{T}{q}r^{7/32}+\frac{T^{1/4}}{q^{1/4}}\right) \ll (qT)^\ve TM\delta^{1/2},
  \end{split}
\end{equation*}
provided $T^{1/8+\ve}\ll M\ll T^{1/2}$. This proves \eqref{eqn: sum over cS} in this case.
So from now on, we can assume $c_2'=h=1$.

\subsection{The case $k=q$}\label{subsec: k=q s=-1}

Now, we show that we can also exclude the case $k=q$.
First note that we can simplify $k=(n_2n_1',q)$. Hence
we distinguish two cases and show that the contribution of $q|n_1'$ and $q|n_2$ is negligible.
We first deal with the case $q|n_1'$.
Denote these terms in \eqref{eqn: cS=} as $\cS_\sigma^\sharp(q,N;\delta)$.
As before, we have
\[
  x \gg \frac{q^2N}{q^3}\left(\frac{\sqrt{\delta N}}{Dq\delta_0}\right)^{-3}
    \gg \frac{q^2 D^3}{(\delta^3N)^{1/2}} \gg q^{1/2-\ve}T^{3/2-\ve} \gg T^\ve.
\]
Note that for $q$ prime, we have $R_q(b)=\varphi(q)$ if $q|b$, and $R_q(b)=-1$ if $q\nmid b$.
By the same process, (writing $n_1'=qn_1''$), we obtain
\begin{equation*}
  \begin{split}
     & \cS_\sigma^\sharp(q,N;\delta)
       = \gamma \sum_{\pm} \sum_{\delta_0\delta'=\delta}\frac{-\mu(\delta_0)\chi(\delta)}{\delta_0}
                  \sum_{\substack{g,n_1'',f_2,r\\ (gn_1''f_2,q\delta)=1\\ (gn_1'',f_2)=1,\mu^2(gn_1'')=1\\(r,q\delta gn_1'' f_2)=1}}
                  \frac{\mu(f_2)\mu(r)}{n_1''\varphi(f_2)}
                  \sum_{\substack{s|g\delta'\\(s,d_2')=1}}\frac{1}{s}
          \sum_{\substack{d_2'\\ (d_2',q\delta gn_1''f_2)=1}} d_2' \\
       & \hskip 20pt \cdot       \sum_{\substack{n_2'\\ (n_2',g\delta'/s)=1}}
                  \frac{A(n_2'rs,n_1''qf_2)}{qn_2'}
            e\left(\pm\sigma\frac{q^2n_1''f_2n_2'\delta_0\overline{d_2'q}}{g\delta'/s}\right)
                    \check{\Psi}_\sigma^\pm\left(\frac{n_2'sN}{(gd_2'q)^3r^2qn_1''f_2},
                        \frac{\sqrt{\delta N}}{q^2gn_1''f_2d_2'r\delta_0}\right).
  \end{split}
\end{equation*}
Now we consider the sum over $d_2'\asymp D_2, n_2'\asymp N_2$,
and by Lemma \ref{lemma: Psi}, we can assume
\begin{equation*}
  \left\{\begin{array}{ll}
    1\leq D_2 \leq \frac{\sqrt{\delta N}}{q^2gn_1''f_2r\delta_0 TM^{1-\ve}}, & \textrm{if } \sigma=1, \\
    1\leq D_2 \asymp \frac{\sqrt{\delta N}}{q^2gn_1''f_2r\delta_0 T}, & \textrm{if } \sigma=-1,
  \end{array}\right.
\end{equation*}
and
\begin{equation*}
  \left\{\begin{array}{ll}
    1\leq N_2 \asymp \frac{(\delta^3N)^{1/2}}{q^2(n_1''f_2)^2 r\delta_0^3 s}, & \textrm{if } \sigma=1, \\
    1\leq N_2 \leq \frac{(\delta^3N)^{1/2}}{q^2(n_1''f_2)^2 r\delta_0^3 s M^{3-\ve}}, & \textrm{if } \sigma=-1.
  \end{array}\right.
\end{equation*}
If $\sigma=-1$, then by Lemma \ref{lemma: Psi}, we have
\begin{equation*}
  \begin{split}
      & \frac{1}{N^{1/2}} \cS_\sigma^\sharp(q,N;\delta)
       \ll  \sum_{\pm} \sum_{\delta_0\delta'=\delta}\frac{1}{\delta_0}
                  \sum_{g,n_1'',f_2,r}  \frac{1}{(gn_1''f_2)^{3/2}r}
                  \sum_{\substack{s|g\delta'\\(s,d_2')=1}}\frac{1}{s^{1/2}}
        \sup_{D_2,N_2} \frac{M}{q^3(D_2N_2)^{1/2}} \\
       & \hskip 5pt \cdot \int_{|t|\ll U} \bigg|
                  \sum_{\substack{d_2'\asymp D_2\\ (d_2',q\delta gn_1''f_2)=1}}
                  \sum_{\substack{n_2'\asymp N_2\\ (n_2',d_2'g\delta'/s)=1}}
                  \alpha(d_2')\beta(n_2')  A(n_2'rs,n_1''qf_2) e\left(\pm\sigma\frac{q^2n_1''f_2n_2'\delta_0\overline{d_2'q}}{g\delta'/s}\right)
                 \left(\frac{n_2'}{d_2'}\right)^{it} \bigg|dt,
  \end{split}
\end{equation*}
for certain coefficients $\alpha(d_2'),\beta(n_2')$ of absolute value at most $1$.
Hence by Lemma \ref{lemma: HLS} again, we have
\begin{equation*}
  \begin{split}
        \frac{1}{N^{1/2}} \cS_\sigma^\sharp(q,N;\delta)
       & \ll  \sup_{D_2,N_2} (qT)^\ve \frac{M}{q^3}r^{7/32}(U+D_2)^{1/2}(\delta U+N_2)^{1/2}\\
       & \ll  (qT)^\ve \delta^{1/2}\left(\frac{T}{q^3}r^{7/32}+\frac{T^{5/4}}{q^{3}M}\right) \ll (qT)^\ve TM\delta^{1/2},
  \end{split}
\end{equation*}
provided $T^{1/8+\ve}\ll M\ll T^{1/2}$.
If $\sigma=1$, we have
\begin{equation*}
  \begin{split}
        \frac{1}{N^{1/2}} \cS_\sigma^\sharp(q,N;\delta)
       & \ll (qT)^\ve \delta^{1/2}\left(\frac{T}{q^3}r^{7/32}+\frac{T^{5/4}M^{1/2}}{q^{3}}\right) \ll (qT)^\ve TM\delta^{1/2},
  \end{split}
\end{equation*}
provided $M\asymp T^{1/2}$.

From now on we assume $(q,n_1')=1$. Since $c_1'=q$ and $n_1'|f_1c_1'$, we have $n_1'|f_1$.
Now we treat the case $q|n_2$.
Denote these terms in \eqref{eqn: cS=} as $\cS_\sigma^{\sharp\sharp}(q,N;\delta)$.
Write $n_2=qn_2'$. By a similar argument, we get
\begin{equation*}
  \begin{split}
      & \cS_\sigma^{\sharp\sharp}(q,N;\delta)
      = \gamma \sum_{\pm} \sum_{\delta_0\delta'=\delta}\frac{\mu(\delta_0)\chi(\delta)}{\delta_0}
                  \sum_{\substack{g,n_1',f_2,r\\ (gn_1'f_2,q\delta)=1\\ (gn_1',f_2)=1,\mu^2(gn_1')=1\\(r,q\delta gn_1'f_2)=1}}
                  \frac{\mu(f_2)\mu(r)}{n_1'\varphi(n_1'f_2)}
                  \sum_{\substack{s|g\delta'\\(s,d_2')=1}} \frac{1}{s}
              \frac{1}{q} \sum_{\substack{d_2'\\(d_2',q\delta gn_1'f_2)=1}} d_2' \\
      & \hskip 30pt \cdot     \sum_{\substack{n_2'\\ (n_2',g\delta'/s)=1}} \frac{A(n_2'rqs,n_1'f_2)}{n_2'}
              e\left(\pm\sigma\frac{n_1'f_2n_2'q\delta_0\overline{d_2'q}}{g\delta'/s}\right)
       \check{\Psi}_\sigma^\pm\left(\frac{n_2'sqN}{(gd_2'q)^3r^2n_1'f_2},\frac{\sqrt{\delta N}}{qgn_1'f_2d_2'r\delta_0}\right).
  \end{split}
\end{equation*}
We consider the sum over $d_2'\asymp D_2, n_2'\asymp N_2$,
and by Lemma \ref{lemma: Psi}, we can assume
\begin{equation*}
  \left\{\begin{array}{ll}
    1\leq D_2 \leq \frac{\sqrt{\delta N}}{qgn_1'f_2r\delta_0 TM^{1-\ve}}, & \textrm{if } \sigma=1, \\
    1\leq D_2 \asymp \frac{\sqrt{\delta N}}{qgn_1'f_2r\delta_0 T}, & \textrm{if } \sigma=-1,
  \end{array}\right.
\end{equation*}
and if $x=\frac{n_2'sqN}{(gd_2'q)^3r^2n_1'f_2}\gg T^\ve$, then
\begin{equation*}
  \left\{\begin{array}{ll}
    1\leq N_2 \asymp \frac{(\delta^3N)^{1/2}}{q(n_1'f_2)^2 r\delta_0^3 s}, & \textrm{if } \sigma=1, \\
    1\leq N_2 \leq \frac{(\delta^3N)^{1/2}}{q(n_1'f_2)^2 r\delta_0^3 s M^{3-\ve}}, & \textrm{if } \sigma=-1.
  \end{array}\right.
\end{equation*}
If $x\gg T^\ve$ and $\sigma=-1$, then by Lemma \ref{lemma: Psi}, we have
\begin{equation*}
  \begin{split}
      &  \frac{1}{N^{1/2}}\cS_\sigma^{\sharp\sharp}(q,N;\delta)
      \ll \sum_{\delta_0\delta'=\delta}\frac{1}{\delta_0}
                  \sum_{g,n_1',f_2,r}
                  \frac{1}{(gf_2)^{3/2}(n_1')^{5/2}r}
                  \sum_{\substack{s|g\delta'\\(s,d_2')=1}} \frac{1}{s^{1/2}}
              \sup_{D_2,N_2} \frac{M}{q^2(D_2N_2)^{1/2}} \\
       & \hskip 10pt \cdot
         \int_{|t|\ll U} \bigg|\sum_{\substack{d_2'\asymp D_2\\(d_2',q\delta gn_1'f_2)=1}}
                  \sum_{\substack{n_2'\asymp N_2\\ (n_2',g\delta'/s)=1}} \alpha(d_2')\beta(n_2')
              A(n_2'rqs,n_1'f_2) e\left(\pm\sigma\frac{n_1'f_2n_2'q\delta_0\overline{d_2'q}}{g\delta'/s}\right)
                \left(\frac{n_2'}{d_2'}\right)^{it} \bigg|dt,
  \end{split}
\end{equation*}
for certain coefficients $|\alpha(d_2')|,|\beta(n_2')|\leq 1$.
The same argument shows that
\[
  \frac{1}{N^{1/2}}\cS_\sigma^{\sharp\sharp}(q,N;\delta)
  \ll (qT)^\ve q^{-1}TM,
\]
provided $T^{1/8+\ve}\ll M\ll T^{1/2}$.
And if $x\gg T^\ve$ and $\sigma=1$, we get
\[
  \frac{1}{N^{1/2}}\cS_\sigma^{\sharp\sharp}(q,N;\delta)
  \ll (qT)^\ve q^{-1}TM,
\]
provided $M\asymp T^{1/2}$ again.
If $x\ll T^\ve$, then by Lemma \ref{lemma: Psi} again, for both $\sigma=\pm1$, we have
\begin{equation*}
  \begin{split}
       \frac{1}{N^{1/2}}\cS_\sigma^{\sharp\sharp}(q,N;\delta)
       & \ll \sum_{\delta_0\delta'=\delta}\frac{1}{\delta_0}
                  \sum_{g,n_1',f_2,r}
                  \frac{1}{(gf_2)^{3/2}(n_1')^{5/2}r}
                  \sum_{\substack{s|g\delta'\\(s,d_2')=1}} \frac{1}{s^{1/2}}
         \sup_{D_2,N_2} \frac{T}{q^2(D_2N_2)^{1/2}} \\
       & \hskip 5pt \cdot  \bigg|\sum_{\substack{d_2'\asymp D_2\\(d_2',q\delta gn_1'f_2)=1}}
                  \sum_{\substack{n_2'\asymp N_2\\ (n_2',g\delta'/s)=1}} \alpha(d_2')\beta(n_2')
       A(n_2'rqs,n_1'f_2) e\left(\pm\sigma\frac{n_1'f_2n_2'q\delta_0\overline{d_2'q}}{g\delta'/s}\right) \bigg|.
  \end{split}
\end{equation*}
A better bound will show up under the assumption $T^{1/8+\ve}\ll M\ll T^{1/2}$.

\subsection{Conclusion}

From the above discussion, we can take $M\asymp T^{1/2}$.
This proves Proposition \ref{prop: q}, and hence Theorem \ref{thm: q}.
On the other hand, recalling $\cR^\pm$ in \eqref{eqn: cR},
we have
\begin{equation}\label{eqn: cR^-}
  \cR^- \ll q^{1/4}TM(qT)^\ve,
\end{equation}
provided $T^{1/8+\ve}\ll M \ll T^{1/2}$.
In the rest of this paper, we will use another method to bound $\cR^+$,
and then prove Theorem \ref{thm: t} and Theorem \ref{thm: main}.

\section{Initial setup of Theorem \ref{thm: t}}\label{sec: setup t}

As in \S\ref{sec: setup q}, we will use
the moment method to prove Theorem \ref{thm: t}.
Since $q^{5/4}T^{3/2}\leq q^{4}T^{4/3}$ if $q\geq T^{2/33}$,
we know that Theorem \ref{thm: t} follows from Theorem \ref{thm: q}
if $q\geq T^{2/33}$. To prove Theorem \ref{thm: t}, we only
need to consider the case $q\leq T^{2/33}$.
However, in the most part of our following arguments, we just
require $q\leq T^{1/4}$, since we will need this in
the proof of Theorem \ref{thm: main},
(see the end of \S\ref{sec: MT}).
Similarly, at first, we will prove the following proposition.

\begin{proposition}\label{prop: t}
  With notation as before, for any $\ve>0$, and $T$ large,
  assuming
  \begin{equation}\label{eqn: q&M}
    q\ll T^{1/6},
    \quad \textrm{and} \quad
    T^{1/3+\ve}\ll M \ll T^{1/2},
  \end{equation}
  we have
  \[
    \sum_{\substack{u_j\in\cB^*(q)\\ T-M\leq t_j\leq T+M}} L(1/2,\phi\times u_j\times\chi)
            + \frac{1}{4\pi}\int_{T-M}^{T+M}|L(1/2+it,\phi\times\chi)|^2dt
    \ll_{\phi,\ve} q^{4}TM(qT)^{\ve}.
  \]
\end{proposition}
%

It's easy to see that Theorem \ref{thm: t} will follow from Proposition \ref{prop: t}.
And as in \S \ref{sec: setup q}, it suffices to prove
\begin{equation}\label{eqn: cR^pm<<}
  \cR^\pm \ll q^{3}TM(qT)^{\ve},
\end{equation}
provided $T^{1/3+\ve}\ll M\ll T^{1/2}$.
Recall that $\cR^\pm$ is defined as in \eqref{eqn: cR}.
Note that we have \eqref{eqn: cR^-}, which gives a better bound
for $\cR^-$. So we only need to prove \eqref{eqn: cR^pm<<} for $\cR^+$.
As Blomer~\cite[\S5]{blomer2012subconvexity} did,
opening the Kloosterman sum, splitting the $n$-sum in to residue classes mod $c$,
and detecting the summation congruence condition
by primitive additive characters, it suffices to prove, for $T^{1/3+\ve}\ll M\ll T^{1/2}$,
(see Blomer~\cite[p. 1406--1407]{blomer2012subconvexity})
\begin{equation}\label{eqn: cR<<}
  \sum_{\substack{m^2\delta^3\leq (qT)^{3+\ve}\\ (\delta,q)=1,\ |\mu(\delta)|=1}}
  \frac{|A(1,m)|}{m\delta^{3/2}}  \sup_N |\cS(q,N;\delta)|
  \ll  q^{3}TM(qT)^\ve,
\end{equation}
where
\begin{equation}\label{eqn: cS(n)}
  \begin{split}
     \cS(q,N;\delta) & := \sum_{q|c}\frac{1}{c^2} \sum_{c_1|c}\underset{b(c_1)}{{\sum}^*} \underset{d(c)}{{\sum}^*}
                e\left(\frac{\bar{d}}{c}\right) \sum_{a(c)}\chi(a)
                e\left(-\frac{\bar{b}a}{c_1}\right) e\left(\frac{\delta da}{c}\right)\\
      & \hskip 90pt \cdot \sum_n A(n,1)e\left(\frac{\bar{b}n}{c_1}\right)
         v\left(\frac{n}{N}\right)n^{-1/2-u}H^+\left(\frac{4\pi \sqrt{\delta n}}{c}\right),
  \end{split}
\end{equation}
where
$v$ a suitable fixed smooth function with support in $[1,2]$, and $N$ satisfying \eqref{eqn: N}.
Note that as pointing out in the end of \S\ref{sec: setup q}, we can restrict the $c$-sum to
$c\leq (qT)^B$, for some fixed $B>0$.
Now we want to use the Voronoi formula to deal with
the $n$-sum in \eqref{eqn: cS(n)}. Before we do this, we need to
give an asymptotic formula for $H^+$.
These will be done in the next section.

\section{Integral transforms and special functions} \label{sec: Psi II}

In this section, we follow Blomer~\cite[\S3]{blomer2012subconvexity},
Li~\cite[\S4]{li2011bounds}, and McKee--Sun--Ye~\cite[\S6]{mckee2015improved} to
give an estimate for the $n$-sum in \eqref{eqn: cS(n)}.
As in Li~\cite[Proposition 4.1]{li2011bounds}, we will
give an asymptotic formula for $H^+$ at first.
We shall follow Li~\cite[\S4]{li2011bounds} and
McKee--Sun--Ye~\cite[\S6]{mckee2015improved} more closely, so readers
who are familiar with their works can safely skip this section at a first reading.
As Li~\cite[\S4]{li2011bounds} did, we have
\begin{equation}\label{eqn: H^+=}
  H^+(x) = H^+_1(x) + H^+_2(x) + O\left(T^{-A}\right),
\end{equation}
where
\begin{equation}\label{eqn: H^+_1}
  H^+_1(x) = \frac{4TM}{\pi} \int_{t=-\infty}^\infty
                \int_{\zeta=- T^\varepsilon}^{T^\varepsilon}
                \frac{1}{\cosh t}\cos(x\cosh\zeta)e\left(\frac{tM\zeta}{\pi}\right)
                e\left(\frac{T\zeta}{\pi}\right)dtd\zeta,
\end{equation}
and
\begin{equation}\label{eqn: H^+_2}
  H^+_2(x) = \frac{4M^2}{\pi} \int_{t=-\infty}^\infty
                \int_{\zeta=- T^\varepsilon}^{T^\varepsilon}
                \frac{t}{\cosh t}\cos(x\cosh\zeta)e\left(\frac{tM\zeta}{\pi}\right)
                e\left(\frac{T\zeta}{\pi}\right)dtd\zeta.
\end{equation}
In the following we only treat $H^+_1(x)$, since $H^+_2(x)$
is a lower order term which can be handled in a similar way.
It is clear that
\begin{equation}\label{eqn: H^+_1=}
  \begin{split}
    H^+_1(x) & =\frac{4MT}{\pi}\int_{-T^\varepsilon}^{T^\varepsilon}
           \widehat{k} \left(-\frac{M\zeta}{\pi}\right) \cos(x\cosh\zeta)e\left(\frac{T\zeta}{\pi}\right)d\zeta \\
     & = 4T\int_{-\frac{MT^\varepsilon}{\pi}}^{\frac{MT^\varepsilon}{\pi}}
     \widehat{k}(\zeta)\cos\left(x\cosh\frac{\zeta\pi}{M}\right)e\left(-\frac{T\zeta}{M}\right)d\zeta,
  \end{split}
\end{equation}
by making a change of variable $-\frac{M\zeta}{\pi}\mapsto\zeta$, here
\begin{equation}\label{eqn: k(t)}
  k(t)=\frac{1}{\cosh t},
\end{equation}
and
\begin{equation}
\widehat{k}(\zeta)=\int_{-\infty}^\infty k(t)e(-t\zeta)dt,
\end{equation}
is its Fourier transform.
Since $\widehat{k}(\zeta)$ is a Schwartz class function, one can extend
the integral in \eqref{eqn: H^+_1=} to $(-\infty, \infty)$ with a negligible error term.
Now let
\begin{equation}\label{eqn: W(x)}
  W(x) := T\int_\mathbb{R} \widehat{k} (\zeta)
  \cos \left(  x\cosh \frac{\zeta \pi}{M}\right)
  e\left(  -\frac{T\zeta}{M}  \right) d\zeta,
\end{equation}
then we have
\[
  H^+_1(x) = 4W(x) + O\left(T^{-A}\right).
\]

\begin{lemma}\label{lemma: W asymp}
  \begin{itemize}
    \item [(i)]  For $|x|\leq T^{1-\ve}M$, we have
            \begin{equation*}
              W(x) \ll _{\ve,A} T^{-A}.
            \end{equation*}
    \item [(ii)] Assume $MT^{1-\ve}\leq x\leq T^2$, and
            $T^{1/3+2\ve}\leq M\leq T^{1/2}$. Let $L_1,L_2\in \bbZ_{+}$.
            We have
            \begin{equation}\label{eqn: H^+ asymp}
                \begin{split}
                    W(x) & = \frac{MT}{\sqrt{x}}\sum_\pm e\left(\mp\frac{x}{2\pi}\pm \frac{T^2}{\pi x}\right)
                        \sum_{l=0}^{L_1}\sum_{0\leq l_1\leq 2l}\sum_{\frac{l_1}{4}\leq l_2\leq L_2} c_{l,l_1,l_2}
                        \frac{M^{2l-l_1}T^{4l_2-l_1}}{x^{l+3l_2-l_1}} \\
                    & \hskip 30pt \cdot \bigg[ \widehat{k}^{(2l-l_1)}\left(\mp\frac{2MT}{\pi x}\right)
                            - \frac{\pi^6 ix}{6!M^6}(y^6\widehat{k}(y))^{(2l-l_1)}\left(\mp\frac{2MT}{\pi x}\right) \\
                    & \hskip 150pt   +  \frac{\pi^{12} i^2x^2}{2!(6!)^2M^{12}}
                        (y^{12}\widehat{k}(y))^{(2l-l_1)}\left(\mp\frac{2MT}{\pi x}\right)\bigg] \\
                    & \hskip 30pt +  O\left(\frac{TM}{\sqrt{x}}\left(\frac{T^4}{x^3}\right)^{L_2+1}
                        + T\left(\frac{M}{\sqrt{x}}\right)^{2L_2+3} + \frac{Tx^3}{M^{18}}\right),
                \end{split}
            \end{equation}
            where $c_{l,l_1,l_2}$ are constants depending only on the indices.
            Here $(y^n\hat{k}(y))^{(l)}(x) = \frac{d^l (y^n\hat{k}(y))}{dy^l}\Big|_{y=x}$.
  \end{itemize}
\end{lemma}

\begin{proof}
  See McKee--Sun--Ye~\cite[Lemma 4.1]{mckee2015improved}.
\end{proof}

Now we estimate $\cS(q,N;\delta)$. Let $x=\frac{4\pi\sqrt{\delta n}}{c}$ in the above lemma.
Assume $MT^{1-\ve}\leq x\leq T^2$.
 By choosing $L_1,L_2$ large enough (depending on $\ve$)
in \eqref{eqn: H^+ asymp}, the contribution to $\cS$ from the first two error terms can be made
as small as desired. We need to estimate the contribution from the last error term.
By the support of $v$, we may assume that $x\ll \frac{(qT)^{3/2+\ve}}{c}$.
Note that for $q\ll T^{1/4}$, we always have $x\leq T^2$.
So the contribution from the last error term is bounded by
\begin{equation}\label{eqn: error1}
  \begin{split}
      &\sum_{q|c}\frac{1}{c^2} \sum_{c_1|c}\underset{b(c_1)}{{\sum}^*}
       \sum_{a(c)}|S(\delta a,1;c)|
      \sum_{n\ll N} \frac{|A(n,1)|}{n^{1/2}} \frac{T(\delta N)^{3/2}}{M^{18}c^3} \\
      &\hskip 80pt \ll (qT)^\ve \frac{T (qT)^{6}}{M^{18}} \sum_{q|c}\frac{1}{c^{5/2}}
       \ll (qT)^\ve q^{7/2} TM \frac{T^6}{M^{19}} \ll (qT)^\ve q^{3/2} TM,
  \end{split}
\end{equation}
provided $T^{1/3+2\ve}\leq M\leq T^{1/2}$ and $q\leq T^{1/6}$.
In the finite series \eqref{eqn: H^+ asymp}, with our assumptions, we always have
\[
  \frac{M^{2l-l_1}T^{4l_2-l_1}}{x^{l+3l_2-l_1}} \ll 1.
\]
All the terms in \eqref{eqn: H^+ asymp} are similar, and can be estimated in a similar way,
so we will only work with the first term, that is, the term with $l=l_1=l_2=0$.
We are led to estimate
\begin{equation}\label{eqn: tilde cS}
  \begin{split}
     \tilde{\cS}(q,N;\delta) & := \frac{TM}{\delta^{1/4}} \sum_{\substack{q|c\\c\ll C}}\frac{1}{c^{3/2}} \sum_{c_1|c}\underset{b(c_1)}{{\sum}^*} \underset{d(c)}{{\sum}^*}
                e\left(\frac{\bar{d}}{c}\right) \sum_{a(c)}\chi(a)
                e\left(-\frac{\bar{b}a}{c_1}\right) e\left(\frac{\delta da}{c}\right)
      \sum_n A(n,1)e\left(\frac{\bar{b}n}{c_1}\right) \psi(n),
  \end{split}
\end{equation}
where
\begin{equation}\label{eqn: C}
  C=\frac{\sqrt{\delta N}}{T^{1-\ve}M},
\end{equation}
and
\begin{equation}\label{eqn: psi}
  \psi(y) = y^{-3/4-u} v\left(\frac{y}{N}\right) \sum_\pm
            e\left(\mp\frac{2\sqrt{\delta y}}{c}\pm \frac{T^2c}{4\pi^2\sqrt{\delta y}}\right)
            \widehat{k}\left(\mp\frac{MTc}{2\pi^2\sqrt{\delta y}}\right).
\end{equation}
Now we apply the Voronoi formula for the $n$-sum in \eqref{eqn: tilde cS},
getting
\begin{equation}\label{eqn: tilde cS Voronoi}
  \begin{split}
     \tilde{\cS}(q,N;\delta) & = \frac{TM}{\delta^{1/4}} \sum_{\substack{q|c\\c\ll C}}\frac{1}{c^{3/2}}
                \sum_{c_1|c}\underset{b(c_1)}{{\sum}^*} \underset{d(c)}{{\sum}^*}
                e\left(\frac{\bar{d}}{c}\right) \sum_{a(c)}\chi(a)
                e\left(-\frac{\bar{b}a}{c_1}\right) e\left(\frac{\delta da}{c}\right)\\
      & \hskip 60pt \cdot \frac{c_1\pi^{3/2}}{2}\sum_{\pm}\sum_{n_1|c_1}\sum_{n_2\geq1}
                \frac{A(n_2,n_1)}{n_1n_2} S\left(b,\pm n_2;c_1/n_1\right)
                \Psi^\pm\left(\frac{n_1^2n_2}{c_1^3}\right) \\
      & = \frac{\pi^{3/2}TM}{2}\sum_{\pm}\sum_{\substack{q|c\\c\ll C}}\frac{1}{c^{3/2}}
                \sum_{c_1|c}c_1 \sum_{n_1|c_1}\sum_{n_2\geq1} \frac{A(n_2,n_1)}{n_1n_2}
                \cT_{c_1,n_1,n_2}^{\pm,\delta}(c,q)  \Psi^\pm\left(\frac{n_1^2n_2}{c_1^3}\right),
  \end{split}
\end{equation}
where $\Psi^\pm(x)$ is defined as in \eqref{eqn: Psi} with $\psi$ in \eqref{eqn: psi},
and $\cT_{c_1,n_1,n_2}^{\pm,\delta}(c,q)=\cT_{\delta,c_1,n_1,n_2}^{\pm,\sigma}(c,q)$ with $\sigma=1$,
where the later one is defined in \eqref{eqn: cT}.
Now, we will deal with $\Psi^\pm(x)$, where $x=\frac{n_1^2n_2}{c_1^3}$.
Since for $q\ll T^{1/4}$, by \eqref{eqn: N} we have
\[
  xN = \frac{n_1^2n_2}{c_1^3}N \geq NC^{-3} \geq M^3 T^{1-\ve}.
\]
By Lemma \ref{lemma: Psi=M+O}, we have
\[
  \Psi^\pm(x) = \gamma_1  x^{2/3} \sum_{\sigma\in\{\pm1\}} \int_0^\infty a_{\sigma}(y)
  e\left(\sigma\frac{2\sqrt{\delta y}}{c}\pm3(xy)^{1/3}\right) dy + \textrm{lower order terms},
\]
where
\begin{equation*}
  a_{\sigma}(y) = y^{-13/12-u} v\left(\frac{y}{N}\right)
            e\left(-\sigma \frac{T^2c}{4\pi^2\sqrt{\delta y}}\right)
            \widehat{k}\left(\sigma\frac{MTc}{2\pi^2\sqrt{\delta y}}\right).
\end{equation*}
Note that for $\Psi^+$, when $\sigma=1$ has no stationary points, so the contribution
to $\tilde{\cS}$ is negligible; so does $\Psi^-$ with $\sigma=-1$.
Hence, we have
\begin{equation*}
  \Psi^\pm(x) = \gamma_1  x^{2/3} \int_0^\infty a_{\mp1}(y)
                e\left(\mp\frac{2\sqrt{\delta y}}{c}\pm3(xy)^{1/3}\right) dy
                + \textrm{lower order terms},
\end{equation*}
By \eqref{eqn: cR<<} and \eqref{eqn: tilde cS}, to prove Proposition \ref{prop: t},
we only need to show
\begin{equation}\label{eqn: tilde cR}
   \begin{split}
     \tilde{\cR}=\tilde{\cR}(q,N;\delta) & := \sum_{\pm}\sum_{\substack{q|c\\c\ll C}}\frac{1}{c^{3/2}}
                \sum_{c_1|c}c_1 \sum_{n_1|c_1}\sum_{n_2\geq1} \frac{A(n_2,n_1)}{n_1n_2}
       \cT_{c_1,n_1,n_2}^{\pm,\delta}(c,q)
                \Psi^\pm_0\left(\frac{n_1^2n_2}{c_1^3}\right)
     \ll (qT)^\ve q^3,
   \end{split}
\end{equation}
where
\begin{equation}\label{eqn: Psi(x)_0}
  \Psi^\pm_0(x) := x^{2/3} \int_0^\infty a_{\mp1}(y)
        e\left(\mp\frac{2\sqrt{\delta y}}{c}\pm3(xy)^{1/3}\right) dy.
\end{equation}

Now we will use the stationary phase method to deal with \eqref{eqn: Psi(x)_0}.
Denote
\begin{equation}\label{eqn: phi(y)}
  \phi(y)=\mp\frac{2\sqrt{\delta y}}{c}\pm3(xy)^{1/3}.
\end{equation}
By the first derivative of $\phi$ and the support of $y$, we know $\Psi^\pm_0(x)$ is negligible unless
\begin{equation}\label{eqn: x&n_2}
  \frac{2}{3}\frac{(\delta^3N)^{1/2}}{c^3} \leq x \leq 2\frac{(\delta^3N)^{1/2}}{c^3}, \quad \textrm{that is,} \quad
  \frac{2}{3}\frac{(\delta^3N)^{1/2}c_1^3}{n_1^2c^3}\leq  n_2 \leq 2\frac{(\delta^3N)^{1/2}c_1^3}{n_1^2c^3}.
\end{equation}
Since the support of $a_{\mp1}$ is in $[N,2N]$, we have
\begin{equation}\label{eqn: int a}
  \int_0^\infty a_{\mp1}(y) e\left(\mp\frac{2\sqrt{\delta y}}{c}\pm3(xy)^{1/3}\right) dy
  = \int_{\frac{1}{4}x^2c^6/\delta^3}^{\frac{9}{2}x^2c^6/\delta^3} a_{\mp1}(y)
    e\left(\mp\frac{2\sqrt{\delta y}}{c}\pm3(xy)^{1/3}\right) dy.
\end{equation}
Note that we have
\[
  N\leq y\leq 2N, \quad \textrm{and} \quad c\ll C=\frac{\sqrt{\delta N}}{T^{1-\ve}M}.
\]
Write $a(y)=a_{\mp1}(y)$. Let $n_0\in\bbN$ which will be chosen later.
Simple calculus estimates give us
\[
  \phi^{(r)}(y) \ll \frac{\sqrt{\delta N}}{c}N^{-r}, \quad
  \textrm{for}\ r=2,\ldots,2n_0+3,
\]
and
\[
  a^{(r)}(y) \ll N^{-13/12}\left(\frac{\delta^{1/2}N^{3/2}}{T^2c}\right)^{-r}, \quad
  \textrm{for}\ r=0,1,\ldots,2n_0+1,
\]
for $y\asymp N$.
There is a stationary phase point $y_0=x^2c^6/\delta^3$ such that $\phi'(y_0)=0$.
To apply Lemma \ref{lemma: MSY}, set
\begin{equation}\label{eqn: MTNU}
  M_0=10N, \quad T_0=\frac{\sqrt{\delta N}}{c}, \quad
  N_0=\frac{\delta^{1/2}N^{3/2}}{T^2c}, \quad U_0=N^{-13/12}.
\end{equation}
Note that $\phi''(y)\gg T_0M_0^{-2}$ for $y\in[\frac{1}{4}x^2 c^6/\delta^3,\frac{9}{2}x^2 c^6/\delta^3]$,
and the condition $N_0\geq M_0^{1+\ve}/\sqrt{T_0}$ is implied by our assumption
$c\ll C$ when $M\gg T^{1/3+2\ve}$.

We are ready to apply  Lemma \ref{lemma: MSY} (where we take $n=n_0$).
The main term of the integral in \eqref{eqn: int a} is
\begin{equation}   \label{eqn: mainterm}
  \frac{e(\phi(y_0)+ {1}/{8})}{\sqrt{\phi^{''}(y_0)}}
  \Big( a(y_0) + \sum_{j=1}^{n_0}\varpi_{2j}\frac{(-1)^{j}(2j-1)!!}{(4\pi i\lambda_2)^j}
  \Big),
\end{equation}
where $\lambda_2 = |\phi''(y_0)|/{2}$.
Notice we have used
\[
  \gamma - \alpha  \asymp \beta - \gamma \asymp  M_0 ,
\]
with $\alpha = \frac{1}{4}x^2 c^6/\delta^3$, $\beta = \frac{9}{2}x^2 c^6/\delta^3$,
and $\gamma = y_0 = x^2 c^6/\delta^3$.
To save time in estimates, notice that there are no boundary terms here.
That is, the terms related to $H_i$ in Lemma \ref{lemma: MSY} will vanish.
This is due to the compact support of $a$, with itself and all of its derivatives zero at
$\frac{1}{4}x^2 c^6/\delta^3$ and $\frac{9}{2}x^2 c^6/\delta^3$.
The sum of the four error terms in Lemma \ref{lemma: MSY} can be simplified to
(using \eqref{eqn: C} and \eqref{eqn: MTNU})
\begin{equation}\label{eqn: ET1}
  \begin{split}
     O\left( \frac{U_0 M_0^{2n_0 +2}}{T_0^{n_0 + 1} N_0^{2n_0 +1}} \right)
    & = O\left( c^{3n_0+2}T^{4n_0+2}N^{-\frac{3}{2}n_0-\frac{13}{12}} \delta^{-\frac{3}{2}n_0-1}\right) \\
    & = O\left( c^{-\frac{1}{2}} \delta^{\frac{1}{4}} N^{\frac{1}{6}} T^{-\frac{1}{2}} M^{-\frac{5}{2}} (TM^{-3})^{n_0} T^\varepsilon \right)
     = O\left( c^{-\frac{1}{2}} \delta^{\frac{1}{4}} N^{\frac{1}{6}} T^{-\frac{3}{2}} \right),
  \end{split}
\end{equation}
provided that $T^{\frac{1}{3}+\ve} \ll M \ll T^{\frac{1}{2}}$ and $n_0 > 1/\ve$.
Note that here we should let the $\ve$ in the upper bound of $C$
be much smaller than the $\ve$ in the lower bound of $M$.
This estimate uses the current assumptions on $c$, and the size of $N$ compared to $q$ and $T$.
Note that
$$
  M_0 \gg N_0, \quad \textrm{if} \quad   q\ll T^{1/4}.
$$

We now need to deal with the $\varpi_{2j}$ terms in \eqref{eqn: mainterm}.
Recall the expression for $\varpi_{2j}$ in equation \eqref{eqn: varpi_k}.
Here we take $2 \leq 2j \leq 2n_0$.
One can see from \eqref{eqn: varpi_k} that the main term from $\varpi_{2j}$ is $a^{(2j)}(y_0)$.
(Here $a$ and $\phi$ take the place of $g$ and $f$ in Lemma \ref{lemma: MSY}.
Further $y_0$ takes the place of $\gamma$.)
Using the above estimates, we have
\begin{equation}
  \varpi_{2j} - \frac{a^{(2j)}(y_0)}{(2j)!} = O \left(\frac{U_0}{M_0 N_0^{2j-1}}\right).
\end{equation}
The constant ultimately depends on $n_0$ and we have used $M_0 \gg N_0$.
To estimate this error term contribution to $\tilde{\mathcal{R}}$,
we must divide by $\lambda_2^{j+ 1/2}$ and sum over $j$. (See (\ref{eqn: mainterm}).)
Since $y_0 \asymp N$, we have
 $\lambda_2 \asymp \frac{\delta^{1/2}}{cN^{3/2}}$.
We have then that this contribution is
(using \eqref{eqn: C} and \eqref{eqn: MTNU})
\begin{equation}\label{eqn: ET2}
  \begin{split}
    O\left( N^{-\frac{25}{12}}
          \left(\frac{T^2 c}{\delta^{1/2}N^{\frac{3}{2}}}\right)^{2j-1}
          \left(\frac{cN^{\frac{3}{2}}}{\delta^{1/2}}\right)^{j + \frac{1}{2}}  \right)
    & =
     O\left(  c^{3j - \frac{1}{2}} T^{4j-2}
     N^{-\frac{3}{2}j + \frac{1}{6}} \delta^{-\frac{3}{2}j+\frac{1}{4}}   \right) \\
    & = O\left( c^{-\frac{1}{2}} \delta^{\frac{1}{4}} N^{\frac{1}{6}} T^{-2}  (TM^{-3})^{j} T^\varepsilon \right) \\
    & = O\left( c^{-\frac{1}{2}} \delta^{\frac{1}{4}} N^{\frac{1}{6}} T^{-\frac{3}{2}} \right).
  \end{split}
\end{equation}

We must now estimate the $a^{(2j)}(y_0)$ term in $\varpi_{2j}$ in \eqref{eqn: mainterm}.
Let $i_1$ be the number of times $v\left(\frac{y}{N}\right)$ is differentiated
plus the number of times $y^{-{13}/{12}-u}$ is differentiated.
So at every differentiation either the factor $\frac{1}{N}$ comes out,
or up to a constant, the factor $\frac{1}{y}$ comes out.
Notice that $\frac{1}{y} \asymp \frac{1}{N}$.
Let $i_2$ be the number of times $\widehat{k}\Big(\frac{MTc}{2 \pi^2 \sqrt{\delta y}}\Big)$
is differentiated, and denote $i_3$ to be the number of times
$e\Big(\frac{-T^2c}{4 \pi^2 \sqrt{\delta y}}\Big)$ is differentiated.
Then $i_1 + i_2 + i_3 = 2j$, and neglecting coefficients (which ultimately depend on $n_0$),
$a^{(2j)}(y_0)$ is the sum over all combinatorial possibilities of
\[
  N^{-\frac{13}{12}-i_1}
  \left( \frac{MTc}{\delta^{1/2}N^{\frac{3}{2}}} \right)^{i_2}
  \left( \frac{T^2 c}{\delta^{1/2}N^{\frac{3}{2}}} \right)^{i_3}.
\]
The main term is when $i_3 = 2j$ and we will estimate this separately, below.
So we can assume in all terms, now, that
$i_1 + i_2 \geq 1$.
To estimate this error term, which is all but one term in $a^{(2j)}(y_0)$, as before, in
\eqref{eqn: mainterm}, we must divide by $\lambda_2^{j + \frac{1}{2}}$ where
$\lambda_2 \asymp \frac{\delta^{1/2}}{cN^{3/2}}$ with our assumption on $y_0$.
We have then a sum of error terms which are all
\[
  O\left( M^{i_2} c^{j + i_2 + i_3 + \frac{1}{2}} T^{i_2 + 2i_3}
  N^{\frac{3}{2}j -i_1 -\frac{3}{2}i_2 -\frac{3}{2}i_3 - \frac{1}{3}}
  \delta^{-\frac{1}{2}(j+i_2+i_3+\frac{1}{2})} \right).
\]
Using $i_3=2j-i_1-i_2$ and $i_1+i_2\geq 1$, the above bound will be
\begin{equation}\label{eqn: ET3}
  \begin{split}
    & O\left( M^{i_2} c^{3j - i_1 + \frac{1}{2}} T^{4j-2i_1-i_2}
  N^{-\frac{3}{2}j -\frac{5}{2}i_1 - \frac{1}{3}}
  \delta^{-\frac{3}{2}j+\frac{1}{2}i_1-\frac{1}{4}} \right) \\
    & \hskip 10pt = O\left( c^{-\frac{1}{2}} \delta^{\frac{1}{4}} N^{\frac{1}{6}} T^{-1} M^{-1} (MT^{-1})^{i_2} (MT^{-1}N^{-3})^{i_1} (TM^{-3})^{j} T^\varepsilon \right)
     = O\left( c^{-\frac{1}{2}} \delta^{\frac{1}{4}} N^{\frac{1}{6}} T^{-\frac{3}{2}} \right).
  \end{split}
\end{equation}
We will bound the contribution of all these error terms
to $\tilde{\cR}$ in the next section.

This leaves the main term of $a^{(2j)}(y_0)$ (where $i_3 = 2j$ and $i_1 = i_2 = 0$) which is
\begin{equation}\label{eqn: a2jalpha}
   a_{2j}(y_0) := \alpha_j  \Big( \frac{T^2c}{\delta^{1/2}y_0^{3/2}} \Big)^{2j}
        v\Big(\frac{y_0}{N}\Big)   \widehat{k}\Big(\frac{MTc}{2 \pi^2 \sqrt{\delta y_0}}\Big)
        e\Big(\frac{-T^2c}{4 \pi^2 \sqrt{\delta y_0}}\Big)  y_0^{-{13}/{12}-u},
\end{equation}
where the constant $\alpha_j$ depends on $j$ which
ultimately can be bounded in terms of $n_0$.
We cannot bound these terms trivially as Li~\cite[\S4]{li2011bounds} did.
Instead, we will apply the Voronoi formula a second time.
This will be done in \S\ref{sec: MT}.

\section{Contribution from the error terms}\label{sec: ET}

In this section, we will bound the contribution
of these error terms \eqref{eqn: ET1}, \eqref{eqn: ET2},
and \eqref{eqn: ET3} to $\tilde{\cR}$, see \eqref{eqn: tilde cR}.
To do this, we need to recall the result of Blomer~\cite{blomer2012subconvexity}
for $\cT_{c_1,n_1,n_2}^{\pm,\delta}(c,q)$.

\begin{lemma}\label{lemma: cT}
  Assume $(q,\delta)=1$, $\delta_0|c_2$, $(r',c_2)=1$. Then
  \begin{equation*}
    \begin{split}
       \cT_{c_1,n_1,n_2}^{\pm,\delta}(c,q) & =
           e\left(\mp\frac{n_1^2n_2(c_2')^2c_2}{c'\delta'}\right)
           \frac{\varphi(c_1)\varphi(c_1/n_1)}{\varphi(c')^2}
           \frac{\mu(\delta_0)\chi(\delta)}{\delta_0}
           r^2 q \chi(-1) \\
         & \hskip 20pt \cdot \sum_{\substack{f_1f_2d_2'=r'\\(d_2',f_1n_1n_2)=1\\
                                    (f_1,f_2)=1,\mu^2(f_1)=1\\(f_1f_2,q)=1,f_2|n_1}}
         \frac{\mu(f_2)}{f_1} e\left(\pm\frac{(n_1'c_2')^2f_2n_2\delta_0\overline{d_2'c_1'}}{f_1\delta'}\right)
         \cV_{c_1,n_1,n_2}^{\pm,\delta}(c,q),
    \end{split}
  \end{equation*}
  where
  \begin{equation*}
    \begin{split}
       \cV_{c_1,n_1,n_2}^{\pm,\delta}(c,q) & := \sum_{g_5,g_6(q)}\chi(g_5g_6)
       \chi(g_5r'+c_2g_6r'\mp c_2n_2n_1c_2')
        \chi(r'g_6\mp n_2n_1c_2')e\left(\bar{\delta'}n_1c_2'\frac{g_5+c_2g_6}{q}\right).
    \end{split}
  \end{equation*}
  Furthermore, we have
  \begin{equation*}
    \begin{split}
       \cV_{c_1,n_1,n_2}^{\pm,\delta}(c,q) & = \chi_h(-1)\frac{h}{\varphi(k)}
            R_k(n_2n_1c_2')R_k(c_2)R_k(n_1c_2')
            H(\mp\overline{r'hk}n_2(n_1c_2')^2c_2\overline{\delta'},\ell).
    \end{split}
  \end{equation*}
  Recall that $h,k,\ell$ are defined in \eqref{eqn: h,k,l}, and
  the relations of the variables are summarized in \eqref{eqn: relations}.
  If one of the conditions $\delta_0|c_2$ and $(r',c_2)=1$ is not
  satisfied, then $\cT_{c_1,n_1,n_2}^{\pm,\delta}(c,q)=0$.
\end{lemma}

\begin{proof}
  See Blomer~\cite[Lemma 12, eq. (50)]{blomer2012subconvexity},
  and the expression for the two exponentials given in the second display
  of~\cite[\S7]{blomer2012subconvexity}.
\end{proof}

By the definition of $\cV_{c_1,n_1,n_2}^{\pm,\delta}(c,q)$,
we have the following trivial bound
$$|\cV_{c_1,n_1,n_2}^{\pm,\delta}(c,q)|\leq q^2.$$
A trivial estimate shows
\[
  \cT_{c_1,n_1,n_2}^{\pm,\delta}(c,q) \ll \frac{\delta_0r^2q^3}{c_2^2n_1}\tau_3(c).
\]
The contribution to \eqref{eqn: tilde cR} from the error terms is bounded by
\[
  \begin{split}
     E & = \sum_{\substack{q|c\\c\ll C}}\frac{1}{c^{\frac{3}{2}}}
     \sum_{c_1|c}c_1 \sum_{n_1|c_1}\sum_{n_2} \frac{|A(n_2,n_1)|}{n_1n_2}
     \frac{\delta_0r^2q^3}{c_2^2n_1}\tau_3(c)
     \left(\frac{n_1^2n_2}{c_1^3}\right)^{\frac{2}{3}}
     c^{-\frac{1}{2}} \delta^{\frac{1}{4}} N^{\frac{1}{6}} T^{-\frac{3}{2}} \\
     & \leq (qT)^\ve T^{-\frac{3}{2}} N^{\frac{1}{6}} \delta^{\frac{1}{4}}
     q \sum_{\substack{q|c\\c\ll C}} \frac{(\delta,c/q)}{c^{2}}
     \sum_{c_1|c} c_1 \sum_{n_1|c_1} \frac{1}{n_1^{\frac{2}{3}}}
     \sum_{n_2} \frac{|A(n_2,n_1)|}{n_2^{\frac{1}{3}}},
  \end{split}
\]
where $n_2$ satisfies \eqref{eqn: x&n_2}.
By \eqref{eqn: C} and \eqref{eqn: N}, we have
\begin{equation}\label{eqn: error2E}
  E \ll (qT)^\ve \delta^{\frac{1}{4}}  N^{\frac{1}{6}} T^{-\frac{3}{2}} (\delta N)^{\frac{1}{3}}
     \sum_{\substack{r\ll C/q}} \frac{(\delta,r)}{r}  \\
    \ll (qT)^\ve (\delta N)^{\frac{1}{2}} T^{-\frac{3}{2}}
    \ll  (qT)^\ve q^{\frac{3}{2}}.
\end{equation}
We finish the estimate of these error terms.

\section{Completion of the proof of Theorem \ref{thm: t} and Theorem \ref{thm: main}}\label{sec: MT}
\label{sec: thm t}

In this section, we will give the proof of Theorem \ref{thm: t} and Theorem \ref{thm: main}.
At first, we will estimate the contribution to $\tilde{\cR}$ of $a_{2j}(y_0)$
in \eqref{eqn: tilde cR}. To bound this, we only need to estimate
\begin{equation}\label{eqn: tilde cR j}
   \begin{split}
     \tilde{\cR}_j =\tilde{\cR}_j(q,N;\delta)
     & := \sum_{\pm}\sum_{\substack{q|c\\c\ll C}}\frac{1}{c^{3/2}}
     \sum_{c_1|c}c_1 \sum_{n_1|c_1}\sum_{n_2\geq1} \frac{A(n_2,n_1)}{n_1n_2}
     x^{2/3} \frac{a_{2j}(y_0)}{\lambda_2^{j+1/2}}
     e(\phi(y_0))\cT_{c_1,n_1,n_2}^{\pm,\delta}(c,q),
   \end{split}
\end{equation}
for all $0\leq j\leq n_0$, where $x=\frac{n_1^2n_2}{c_1^3}$, $y_0=\frac{x^2c^6}{\delta^3}$,
and $\lambda_2=\frac{\delta^{1/2}}{12cy_0^{3/2}}$.
Inserting these values,
together with \eqref{eqn: a2jalpha}, we have
\begin{equation}
   \begin{split}
     \tilde{\cR}_j & = 12^{j+1/2}\alpha_j T^{4j} \delta^{3j+3/4+3u}
     \sum_{\pm}\sum_{\substack{q|c\\c\ll C}}\frac{1}{c^{6j+3+6u}}
     \sum_{c_1|c}c_1^{9j+1+6u} \sum_{n_1|c_1}\frac{1}{n_1^{6j+1+4u}}
      \\
     & \hskip 60pt \cdot \sum_{n_2\geq1} \frac{A(n_2,n_1)}{n_2^{3j+1+2u}}
        v\Big(\frac{n_1^4n_2^2c^6}{c_1^6\delta^3N}\Big)
        \widehat{k}\Big(\frac{MT\delta c_1^3}{2 \pi^2 n_1^2n_2c^2}\Big)
        e\Big(\frac{-T^2\delta c_1^3}{4 \pi^2 n_1^2n_2c^2}\Big)
        \check{\cT}_{c_1,n_1,n_2}^{\pm,\delta}(c,q),
   \end{split}
\end{equation}
where
\begin{equation}
  \check{\cT}_{c_1,n_1,n_2}^{\pm,\delta}(c,q):=
  e\Big(\pm\frac{n_1^2n_2c^2}{c_1^3\delta}\Big)\cT_{c_1,n_1,n_2}^{\pm,\delta}(c,q).
\end{equation}
By Lemma \ref{lemma: cT}, after some simplification, we have
\begin{equation}\label{eqn: tilde cR_j=}
  \begin{split}
     \tilde{\cR}_j
     & = \frac{\gamma_j T^{4j} \delta^{3j+3/4+3u}}{q}
        \sum_{\pm}\sum_{\delta_0\delta'=\delta}
        \frac{\mu(\delta_0)\chi(\delta)}{\delta_0^{6j+2+6u}}
        \sum_{c_1'c_2'=q} \frac{(c_1')^{3j}}{(c_2')^{6j+1+6u}} \\
     & \hskip 20pt \cdot \sum_{\substack{f_1f_2d_2'\ll C/(q\delta_0) \\ (f_1f_2d_2',c_2'\delta)=1
                         \\(f_1,f_2)=1,\mu^2(f_1)=1\\(f_1f_2,qd_2')=1}}
         \frac{\mu(f_2) f_1^{3j} (d_2')^{3j}}{f_2^{3j+4u}f_1f_2}
        \sum_{\substack{n_1'|f_1c_1'\\(n_1',d_2')=1}} \frac{1}{(n_1')^{6j+1+4u}}
        \frac{\varphi(f_1f_2d_2'c_1')\varphi(f_1d_2'c_1'/n_1')}{\varphi(f_1f_2d_2'c_1'c_2')^2} \\
     & \hskip 40pt \cdot  \sum_{\substack{n_2\\(n_2,d_2')=1}}
        \frac{A(n_2,n_1'f_2)}{n_2^{3j+1+2u}}
        e\left(\pm\frac{(n_1'c_2')^2f_2n_2\delta_0\overline{d_2'c_1'}}{f_1\delta'}\right)
        \frac{h\chi_h(-1)}{\varphi(k)}  \\
     & \hskip 60pt \cdot   R_k(n_2n_1'f_2c_2')R_k(c_2'\delta_0)R_k(n_1'f_2c_2')
        H(\mp\overline{f_1d_2'hk}n_2(n_1'c_2')^2f_2c_2'\delta_0\overline{\delta'},\ell) \\
     & \hskip 80pt \cdot  v\Big(\frac{n_1^4n_2^2(c_2')^6\delta_0^3}{(\delta')^3N}\Big)
        \widehat{k}\Big(\frac{MTc_1'f_1d_2'\delta'}{2 \pi^2 \delta_0f_2(c_2'n_1')^2n_2}\Big)
        e\Big(\frac{-T^2c_1'f_1d_2'\delta'}{4 \pi^2  \delta_0f_2(c_2'n_1')^2n_2}\Big),
   \end{split}
\end{equation}
where $\gamma_j=12^{j+\frac{1}{2}}\chi(-1)\alpha_j$.
Recall that the relations of the new variables and the old variables are
\begin{equation}\label{eqn: var}
  c=c_1'f_1f_2d_2'c_2'\delta_0, \quad
  c_1=c_1'f_1f_2d_2',\quad
  c_2=c_2'\delta_0,\quad
  r=f_1f_2d_2'\delta_0,\quad
  n_1=n_1'f_2.
\end{equation}
Note that $u\in[\ve-i\log^2(qT),\ve+i\log^2(qT)]$, so the appearance
of $u$ in the exponents is harmless.
As in \S\ref{sec: thm q}, we have several cases to handle.
Since all these cases are similar, we will only deal with the main
case, that is
%
%
the case $c_1'=q,\ c_2'=h=k=1$.
Denote these terms in \eqref{eqn: tilde cR_j=} as $\tilde{\cR}_j^\dag$.
Note that by \eqref{eqn: h,k,l} we have $(d_2'n_1'n_2,q)=1$. Write $f_1=n_1'g$.  Then we have
\begin{equation}\label{eqn: R_j^d}
  \begin{split}
     \tilde{\cR}_j^\dag
     & = \gamma_j q^{3j-1} T^{4j} \delta^{3j+3/4+3u}
        \sum_{\pm}\sum_{\delta_0\delta'=\delta}
        \frac{\mu(\delta_0)\chi(\delta)}{\delta_0^{6j+2+6u}}
        \sum_{\substack{gn_1'f_2\ll C/(q\delta_0) \\ (gn_1'f_2,q\delta)=1
                         \\(gn_1',f_2)=1,\mu^2(gn_1')=1}}
         \\
     & \hskip 30pt \cdot  \frac{\mu(f_2) g^{3j-1} }{f_2^{3j+1+4u}\varphi(n_1'f_2)(n_1')^{3j+2+4u}}
        \sum_{\substack{d_2'\ll C/(q\delta_0gn_1'f_2)\\(d_2',qgn_1'f_2\delta)=1}}
        (d_2')^{3j} \\
     & \hskip 60pt \cdot  \sum_{\substack{n_2\\(n_2,d_2')=1}}
        \frac{A(n_2,n_1'f_2)}{n_2^{3j+1+2u}}
        e\left(\pm\frac{n_1'f_2\delta_0\overline{d_2'q}n_2}{g\delta'}\right)
        H(\mp\overline{gd_2'}n_1'f_2\delta_0\overline{\delta'}n_2,q) \\
     & \hskip 90pt \cdot   v\Big(\frac{(n_1'f_2)^4\delta_0^3n_2^2}{(\delta')^3N}\Big)
        \widehat{k}\Big(\frac{MTqgd_2'\delta'}{2 \pi^2 \delta_0f_2n_1'n_2}\Big)
        e\Big(\frac{-T^2qgd_2'\delta'}{4 \pi^2 \delta_0f_2n_1'n_2}\Big).
   \end{split}
\end{equation}
To remove the coprime condition $(n_2,d_2')=1$,
we split the $n_2$-sum into residue classes mod $d_2'$,
and then detect the summation congruence condition by
additive characters mod $d_2'$, getting
\begin{equation*}
  \begin{split}
     \tilde{\cR}_j^\dag
     & \ll (qT)^\ve q^{3j-1} T^{4j} \delta^{3j+3/4}
        \sum_{\pm}\sum_{\delta_0\delta'=\delta} \frac{1}{\delta_0^{6j+2}}
        \sum_{\substack{gn_1'f_2\ll C/(q\delta_0) \\ (gn_1'f_2,q\delta)=1
                         \\(gn_1',f_2)=1,\mu^2(gn_1')=1}}
         \\
     & \hskip 30pt \cdot \frac{g^{3j-1}}{f_2^{3j+2}(n_1')^{3j+3}}
        \sum_{\substack{d_2'\ll C/(q\delta_0gn_1'f_2)\\(d_2',qgn_1'f_2\delta)=1}}
        (d_2')^{3j}  \Bigg|
        \underset{a_1(d_2')}{{\sum}^*} \frac{1}{d_2'}
        \sum_{b_1(d_2')}e\left(-\frac{b_1a_1}{d_2'}\right)\\
     & \hskip 60pt \cdot \sum_{n_2}
        \frac{A(n_2,n_1'f_2)}{n_2^{3j+1+2u}} e\left(\frac{b_1n_2}{d_2'}\right)
        e\left(\pm\frac{n_1'f_2\delta_0\overline{d_2'q}n_2}{g\delta'}\right)
        H(\mp\overline{gd_2'}n_1'f_2\delta_0\overline{\delta'}n_2,q) \\
     & \hskip 90pt \cdot  v\Big(\frac{(n_1'f_2)^4\delta_0^3n_2^2}{(\delta')^3N}\Big)
        \widehat{k}\Big(\frac{MTqgd_2'\delta'}{2 \pi^2 \delta_0f_2n_1'n_2}\Big)
        e\Big(\frac{-T^2qgd_2'\delta'}{4 \pi^2 \delta_0f_2n_1'n_2}\Big)\Bigg|.
   \end{split}
\end{equation*}
Now by \eqref{eqn: H(w;q)}, \eqref{eqn: H to H^*}, and \eqref{eqn: H^*=},
we split on $q_2=1$ or $q_2=q$, getting
\begin{equation}\label{eqn: cR_j^d=}
  \tilde{\cR}_j^\dag \ll \tilde{\cR}_j^{\dag,1}+\tilde{\cR}_j^{\dag,2},
\end{equation}
where
\begin{equation}\label{eqn: cR_j^d1}
  \begin{split}
     \tilde{\cR}_j^{\dag,1}
     & := (qT)^\ve q^{3j-1} T^{4j} \delta^{3j+3/4}
        \sum_{\pm}\sum_{\delta_0\delta'=\delta} \frac{1}{\delta_0^{6j+2}}
        \sum_{\substack{gn_1'f_2\ll C/(q\delta_0) \\ (gn_1'f_2,q\delta)=1
                         \\(gn_1',f_2)=1,\mu^2(gn_1')=1}}
         \\
     & \hskip 30pt \cdot \frac{g^{3j-1} }{f_2^{3j+2}(n_1')^{3j+3}}
        \sum_{\substack{d_2'\ll C/(q\delta_0gn_1'f_2)\\(d_2',qgn_1'f_2\delta)=1}}
        (d_2')^{3j}
        \frac{1}{d_2'}  \sum_{b_1(d_2')} |S(0,-b_1;d_2')| \\
     & \hskip 60pt \cdot \Bigg| \sum_{n_2}
        \frac{A(n_2,n_1'f_2)}{n_2^{3j+1+2u}} e\left(\frac{b_1n_2}{d_2'}\right)
        e\left(\pm\frac{n_1'f_2\delta_0\overline{d_2'q}n_2}{g\delta'}\right)\\
     & \hskip 90pt \cdot  v\Big(\frac{(n_1'f_2)^4\delta_0^3n_2^2}{(\delta')^3N}\Big)
        \widehat{k}\Big(\frac{MTqgd_2'\delta'}{2 \pi^2 \delta_0f_2n_1'n_2}\Big)
        e\Big(\frac{-T^2qgd_2'\delta'}{4 \pi^2 \delta_0f_2n_1'n_2}\Big)\Bigg|,
   \end{split}
\end{equation}
and
\begin{equation}\label{eqn: cR_j^d2}
  \begin{split}
     \tilde{\cR}_j^{\dag,2}
     & := (qT)^\ve q^{3j+\frac{1}{2}} T^{4j} \delta^{3j+3/4}
        \sum_{\pm}\sum_{\delta_0\delta'=\delta} \frac{1}{\delta_0^{6j+2}}
        \sum_{\substack{gn_1'f_2\ll C/(q\delta_0) \\ (gn_1'f_2,q\delta)=1
                         \\(gn_1',f_2)=1,\mu^2(gn_1')=1}}
         \\
     & \hskip 30pt \cdot \frac{g^{3j-1} }{f_2^{3j+2}(n_1')^{3j+3}}
        \sum_{\substack{d_2'\ll C/(q\delta_0gn_1'f_2)\\(d_2',qgn_1'f_2\delta)=1}}
        (d_2')^{3j}
        \frac{1}{d_2'}  \sum_{b_1(d_2')} |S(0,-b_1;d_2')| \\
     & \hskip 60pt \cdot \frac{1}{\varphi(q)}\sum_{\psi(q)} \Bigg| \sum_{n_2}
        \frac{A(n_2,n_1'f_2)}{n_2^{3j+1+2u}} e\left(\frac{b_1n_2}{d_2'}\right)
        e\left(\pm\frac{n_1'f_2\delta_0\overline{d_2'q}n_2}{g\delta'}\right)
        \psi(n_2)\\
     & \hskip 120pt \cdot  v\Big(\frac{(n_1'f_2)^4\delta_0^3n_2^2}{(\delta')^3N}\Big)
        \widehat{k}\Big(\frac{MTqgd_2'\delta'}{2 \pi^2 \delta_0f_2n_1'n_2}\Big)
        e\Big(\frac{-T^2qgd_2'\delta'}{4 \pi^2 \delta_0f_2n_1'n_2}\Big)\Bigg|.
   \end{split}
\end{equation}
We will focus on $\tilde{\cR}_j^{\dag,2}$, since it turns out that
$\tilde{\cR}_j^{\dag,1}$ is easier and has a better upper bound.
At first we need to remove the factor $\psi(n_2)$ in the innermost
sum of \eqref{eqn: cR_j^d2}. Again, we
split the $n_2$-sum into residue classes mod $q$,
and then detect the summation congruence condition by
additive characters mod $q$, getting
\begin{equation*}
  \begin{split}
     \tilde{\cR}_j^{\dag,2}
     & = (qT)^\ve q^{3j+\frac{1}{2}} T^{4j} \delta^{3j+3/4}
        \sum_{\pm}\sum_{\delta_0\delta'=\delta} \frac{1}{\delta_0^{6j+2}}
        \sum_{\substack{gn_1'f_2\ll C/(q\delta_0) \\ (gn_1'f_2,q\delta)=1
                         \\(gn_1',f_2)=1,\mu^2(gn_1')=1}}
        \frac{g^{3j-1} }{f_2^{3j+2}(n_1')^{3j+3}} \\
     & \hskip 30pt \cdot \sum_{\substack{d_2'\ll C/(q\delta_0gn_1'f_2)\\(d_2',qgn_1'f_2\delta)=1}}
        \frac{(d_2')^{3j}}{d_2'}  \sum_{b_1(d_2')} |S(0,-b_1;d_2')|
        \frac{1}{\varphi(q)}\sum_{\psi(q)} \frac{1}{q}\sum_{b_2(q)} |S_\psi(0,-b_2;q)| \\
     & \hskip 60pt \cdot  \Bigg|  \sum_{n_2}
        \frac{A(n_2,n_1'f_2)}{n_2^{3j+1+2u}} e\left(\frac{b_1n_2}{d_2'}\right)e\left(\frac{b_2n_2}{q}\right)
         e\left(\pm\frac{n_1'f_2\delta_0\overline{d_2'q}n_2}{g\delta'}\right)\\
     & \hskip 120pt \cdot
        v\Big(\frac{(n_1'f_2)^4\delta_0^3n_2^2}{(\delta')^3N}\Big)
        \widehat{k}\Big(\frac{MTqgd_2'\delta'}{2 \pi^2 \delta_0f_2n_1'n_2}\Big)
        e\Big(\frac{-T^2qgd_2'\delta'}{4 \pi^2 \delta_0f_2n_1'n_2}\Big)\Bigg|,
   \end{split}
\end{equation*}
where
\[
  S_\psi(m,n;c)=\underset{d(c)}{{\sum}^*} \psi(d) e\left(\frac{m\bar{d}+nd}{c}\right)
\]
is the Kloosterman sum with character $\psi$.
Note that $S(0,-b_1;d_2')$ is related to the Ramanujan sum,
and $S_\psi(0,-b_2;q)$ is related to the Gauss sum,
inserting the upper bound for these sums
\[
  |S(0,-b_1;d_2')|\ll (\gcd(b_1,d_2'))^{1+\varepsilon}
  \quad \textrm{and} \quad
  |S_\psi(0,-b_2;q)| \leq \sqrt{q},
\]
we have
\begin{equation}\label{eqn: cR_j^d2<<}
  \begin{split}
     & \tilde{\cR}_j^{\dag,2}
       \ll (qT)^\ve q^{3j+1} T^{4j} \delta^{3j+3/4}
        \sum_{\pm}\sum_{\delta_0\delta'=\delta} \frac{1}{\delta_0^{6j+2}}
        \sum_{\substack{gn_1'f_2\ll C/(q\delta_0) \\ (gn_1'f_2,q\delta)=1
                         \\(gn_1',f_2)=1,\mu^2(gn_1')=1}}
        \frac{g^{3j-1} }{f_2^{3j+2}(n_1')^{3j+3}}\\
     & \hskip 50pt \cdot
        \sum_{\substack{d_2'\ll C/(q\delta_0gn_1'f_2)\\(d_2',qgn_1'f_2\delta)=1}}
        \frac{(d_2')^{3j}}{d_2'}  \sum_{b_1(d_2')} (b_1,d_2')
        \frac{1}{q}\sum_{b_2(q)}
        \bigg|  \sum_{n_2}
        \frac{A(n_2,n_1'f_2)}{n_2^{3j+1+2u}} \\
     & \hskip 5pt \cdot
     e\left(\frac{(b_1qg\delta'+b_2d_2'g\delta'\pm n_1'f_2\delta_0)n_2}{d_2'qg\delta'}\right)
        v\Big(\frac{(n_1'f_2)^4\delta_0^3n_2^2}{(\delta')^3N}\Big)
        \widehat{k}\Big(\frac{MTqgd_2'\delta'}{2 \pi^2 \delta_0f_2n_1'n_2}\Big)
        e\Big(\frac{-T^2qgd_2'\delta'}{4 \pi^2 \delta_0f_2n_1'n_2}\Big)\bigg|.
   \end{split}
\end{equation}

Now we will handle the inner $n_2$-sum, that is,
\begin{equation}\label{eqn: n_2-sum}
  \sum_{n_2} A(n_2,n_1'f_2)  e\left(\frac{b'n_2}{c'}\right) w_j(n_2),
\end{equation}
where
\begin{equation}\label{eqn: w_j}
  w_j(y) := \frac{1}{y^{3j+1+2u}} v\Big(\frac{(n_1'f_2)^4\delta_0^3y^2}{(\delta')^3N}\Big)
        \widehat{k}\Big(\frac{MTqgd_2'\delta'}{2 \pi^2 \delta_0f_2n_1'y}\Big)
        e\Big(\frac{-T^2qgd_2'\delta'}{4 \pi^2 \delta_0f_2n_1'y}\Big),
\end{equation}
and
\begin{equation}\label{eqn: b'/c'}
  \frac{b'}{c'} := \frac{b_1qg\delta'+b_2d_2'g\delta'\pm n_1'f_2\delta_0}{d_2'qg\delta'},
  \quad \textrm{with}\quad
  (b',c')=1, \quad \textrm{and}\quad
  c'| d_2'qg\delta'.
\end{equation}
We apply the Voronoi formula on $GL(3)$ a second time, getting
\begin{equation}\label{eqn: VSF2}
  \begin{split}
     & \sum_{n_2} A(n_2,n_1'f_2)  e\left(\frac{b'n_2}{c'}\right) w_j(n_2) \\
       & \hskip 30pt = \frac{c'\pi^{3/2}}{2} \sum_{\pm} \sum_{l_1|c'n_1'f_2} \sum_{l_2=1}^{\infty}
              \frac{A(l_2,l_1)}{l_1l_2} S\left(n_1'f_2\bar{b'},\pm l_2;\frac{n_1'f_2c'}{l_1}\right)
              \cW_j^{\pm}\left(\frac{l_1^2l_2}{(c')^3n_1'f_2}\right),
  \end{split}
\end{equation}
where $\cW_j^\pm$ is defined by \eqref{eqn: Psi} with $\psi=w_j$.
By the support of $v$, we know $w_j$ is supported in
$\big[Y,\sqrt{2}Y\big]$, with
$$
  Y:=\frac{\sqrt{\delta^3 N}}{(n_1'f_2)^2\delta_0^3} \geq 1.
$$
By the facts $c'\leq d_2'qg\delta'$ and $q\delta_0gn_1'f_2d_2'\ll C$,
the bounds for $N$ and $C$, i.e., \eqref{eqn: N} and \eqref{eqn: C},
and the bounds for $q$ and $M$, i.e., \eqref{eqn: q&M},
writing $x=\frac{l_1^2l_2}{(c')^3n_1'f_2}$,
we have
\[
  xY=\frac{l_1^2l_2}{(c')^3n_1'f_2} \frac{\sqrt{\delta^3 N}}{(n_1'f_2)^2\delta_0^3}
  \gg \frac{\sqrt{\delta^3 N}}{C^3(\delta')^3} \gg \frac{M^3T^{-\ve}}{q^3} \gg T^{\ve}.
\]
Now by Lemma \ref{lemma: Psi=M+O}, we have
\[
    \cW_j^\pm(x) = x\int_0^\infty w_j(y) \sum_{\ell=1}^{K} \frac{\gamma_\ell}{(xy)^{\ell/3}}
    e\left(\pm3(xy)^{1/3}\right) dy + O\left(T^{-A}\right),
\]
for some large $K$ and $A$.
We will only deal with the term with $\ell=1$,
since the others can be handled similarly.
By \eqref{eqn: w_j}, we are led to estimate
\begin{equation*}
      \cW_{j,0}^+(x) := x^{2/3} \int_0^\infty b(y)  e\left(\phi_1(y)\right) dy,
      \quad \textrm{and}\quad
      \cW_{j,0}^-(x) := x^{2/3} \int_0^\infty b(y)  e\left(\phi_2(y)\right) dy,
\end{equation*}
where
\[
  b(y) := \frac{1}{y^{3j+\frac{4}{3}+2u}}
        v\Big(\frac{(n_1'f_2)^4\delta_0^3y^2}{(\delta')^3N}\Big)
        \widehat{k}\Big(\frac{MTqgd_2'\delta'}{2 \pi^2 \delta_0f_2n_1'y}\Big),
\]
and
\[
     \phi_1(y) := -\frac{T^2qgd_2'\delta'}{4 \pi^2 \delta_0f_2n_1'y} + 3(xy)^{1/3},
      \quad \textrm{and}\quad
     \phi_2(y) := -\frac{T^2qgd_2'\delta'}{4 \pi^2 \delta_0f_2n_1'y} - 3(xy)^{1/3}.
\]
By the support of $v$, we have
\begin{equation*}
  \int_0^\infty b(y)  e\left(\phi_i(y)\right) dy
  = \int_{Y}^{\sqrt{2}Y} b(y)  e\left(\phi_i(y)\right) dy,
  \quad \textrm{for } i=1,2.
\end{equation*}
For $y\in[Y,\sqrt{2}Y]$, we have
\[
  |\phi_i^{(r)}(y)| \leq C_r T_1/M_1^{r}, \quad
  |b^{(s)}(y)| \leq C_s U_1/N_1^{s},
\]
where
\[
  T_1 = \max\left(\frac{T^2qgd_2'\delta'}{\delta_0f_2n_1'Y},(xY)^{1/3}\right),\quad
  M_1 = Y,\quad
  U_1 = \frac{1}{Y^{3j+4/3}},\quad
  N_1 = Y.
\]
Since we have
\[
  \phi_1'(y) = \frac{T^2qgd_2'\delta'}{4 \pi^2 \delta_0f_2n_1'y^2} + x^{1/3}y^{-2/3} \asymp T_1/M_1,
\]
by partial integration $r$ times, we have
\[
  \cW_{j,0}^+(x) \ll x^{2/3}\frac{U_1}{T_1^r} \ll \frac{x^{2/3}}{T_1^2} \frac{U_1}{T_1^{r-2}} \ll T^{-A},
\]
for sufficient large $r$. So $\cW_{j,0}^+(x)$ is negligible.

Now we turn to $\cW_{j,0}^-(x)$.  Since
\[
  \phi_2'(y) = \frac{T^2qgd_2'\delta'}{4 \pi^2 \delta_0f_2n_1'y^2} - x^{1/3}y^{-2/3},
\]
if
\begin{equation}\label{eqn: x><}
  x\geq \frac{T^6(qgd_2'n_1'f_2\delta_0)^3(n_1'f_2)^2}{10\pi^6 (\delta')^3 N^2},
  \quad \textrm{or}\quad
  x\leq \frac{T^6(qgd_2'n_1'f_2\delta_0)^3(n_1'f_2)^2}{1100\pi^6 (\delta')^3 N^2},
\end{equation}
one has
\[
  |\phi_2'(y)| \asymp T_1/M_1.
\]
By the same argument, we show $\cW_{j,0}^-(x)$ is negligible.
For the remaining case
\begin{equation}\label{eqn: x&l_2}
  \frac{T^6(qgd_2'n_1'f_2\delta_0)^3(n_1'f_2)^2}{1100\pi^6 (\delta')^3 N^2}
  \leq x \leq
  \frac{T^6(qgd_2'n_1'f_2\delta_0)^3(n_1'f_2)^2}{10\pi^6 (\delta')^3 N^2},
  \quad \textrm{i.e.}\quad
  \frac{L_2}{1100}\leq  l_2 \leq \frac{L_2}{10},
\end{equation}
with
\[
  L_2 = \frac{T^6(qgd_2'n_1'f_2\delta_0)^3(n_1'f_2)^3(c')^3}{\pi^6 (\delta')^3 N^2 l_1^2},
\]
we have
\[
  |\phi_2''(y)| \gg T_1/M_1^2,
\]
for any $y\in[Y,\sqrt{2}Y]$.
Note that in this case we have
\[
  T_1\asymp \frac{T^2qgd_2'\delta'}{\delta_0f_2n_1'Y},
  \quad \textrm{and}\quad
  x\asymp T_1^3/Y.
\]
Therefore, by the second derivative test (Lemma \ref{lemma: SP})
we have
\begin{equation}\label{eqn: cW_j<<}
  \cW_{j,0}^-(x) \ll \frac{x^{2/3} U_1}{(T_1/M_1^2)^{1/2}}
  \ll \frac{T_1^{3/2}}{Y^{3j+1}}
  \ll T^3 (qgd_2')^{3/2}(n_1'f_2)^{6j+\frac{7}{2}}\delta_0^{9j+\frac{9}{2}}
        \delta^{\frac{3}{2}}(\delta^3N)^{-\frac{3}{2}j-\frac{5}{4}}.
\end{equation}
Combining \eqref{eqn: cR_j^d2<<}, \eqref{eqn: VSF2}, and \eqref{eqn: cW_j<<},
and invoking the trivial bound for the Kloosterman sum in \eqref{eqn: VSF2},
one concludes that
\begin{equation*}
  \begin{split}
     \tilde{\cR}_j^{\dag,2}
     & \ll (qT)^\ve q^{3j+1} T^{4j} \delta^{3j+3/4}
        \sum_{\pm}\sum_{\delta_0\delta'=\delta} \frac{1}{\delta_0^{6j+2}}
        \sum_{\substack{gn_1'f_2\ll C/(q\delta_0) \\ (gn_1'f_2,q\delta)=1
                         \\(gn_1',f_2)=1,\mu^2(gn_1')=1}}
        \frac{g^{3j-1} }{f_2^{3j+2}(n_1')^{3j+3}}\\
     & \hskip 90pt \cdot  \sum_{\substack{d_2'\ll C/(q\delta_0gn_1'f_2)\\(d_2',qgn_1'f_2\delta)=1}}
        \frac{(d_2')^{3j}}{d_2'}  \sum_{b_1(d_2')} (b_1,d_2')
        \frac{1}{q}\sum_{b_2(q)}
        c' \sum_{l_1|c'n_1'f_2} \sum_{l_2=1}^{\infty}
        \frac{|A(l_2,l_1)|}{l_1l_2} \\
     & \hskip 150pt \cdot  \frac{n_1'f_2c'}{l_1}
        T^3 (qgd_2')^{3/2}(n_1'f_2)^{6j+\frac{7}{2}}\delta_0^{9j+\frac{9}{2}}
        \delta^{\frac{3}{2}}(\delta^3N)^{-\frac{3}{2}j-\frac{5}{4}} \\
     & \ll (qT)^\ve q^{3j+9/2} T^{4j+3} N^{-3j/2-5/4} \delta^{-3j/2+1/2}
         \\
     & \hskip 90pt \cdot \sum_{gn_1'f_2\delta_0d_2'\ll C/q}
        \frac{(gn_1'f_2\delta_0d_2')^{3j+7/2}}{gf_2(n_1')^{2}\delta_0}
        \sum_{l_1|c'n_1'f_2}\frac{1}{l_1^2} \sum_{l_2=1}^{\infty} \frac{|A(l_2,l_1)|}{l_2}.
   \end{split}
\end{equation*}
By \eqref{eqn: blomer}, we have
\begin{equation*}
  \tilde{\cR}_j^{\dag,2}
  \ll (qT)^\ve C^{3j+9/2} T^{4j+3} N^{-3j/2-5/4}  \delta^{-3j/2+1/2}.
\end{equation*}
And by \eqref{eqn: C} and \eqref{eqn: N}, we have
\begin{equation}\label{eqn: MT<<}
  \tilde{\cR}_j^{\dag,2}
  \ll (qT)^\ve \delta^{11/4}NT^{j-3/2}M^{-3j-9/2}
  \ll (qT)^\ve q^3 T^{j+3/2} M^{-3j-9/2}
  \ll (qT)^\ve q^{3},
\end{equation}
provided $T^{1/3+\ve}\leq M\leq T^{1/2}$.
This proves Proposition \ref{prop: t}, and hence, Theorem \ref{thm: t}.

\begin{proof}[Proof of Theorem \ref{thm: main}]
  One can use Theorem \ref{thm: q} and Theorem \ref{thm: t} directly to
  give a hybrid subconvexity bound with $\theta=1/70$.
  To get a better bound, that is, to prove Theorem \ref{thm: main}
  with $\theta=(35-\sqrt{1057})/56$, which we fix from now on,
  we will modify the proof of Theorem \ref{thm: t}.
  At first, note that if $q\geq T^{\theta/(1/4-\theta)}=T^{(\sqrt{1057}-23)/44}$,
  then
  \[ q^{5/4}T^{3/2} \leq (qT)^{3/2-\theta}. \]
  Hence in this case, Theorem \ref{thm: main} follows from Theorem \ref{thm: q}.
  Now we assume
  \begin{equation}\label{eqn: q Delta}
    q \leq T^{\theta/(1/4-\theta)}=T^{(\sqrt{1057}-23)/44} < T^{1/4}.
  \end{equation}
  As in \S\ref{sec: setup t}, we only need to prove
  \[
    \sum_{\substack{u_j\in\cB^*(q)\\ T-M\leq t_j\leq T+M}} L(1/2,\phi\times u_j\times\chi)
            + \frac{1}{4\pi}\int_{T-M}^{T+M}|L(1/2+it,\phi\times\chi)|^2dt
    \ll_{\phi,\ve} q^{3/2-\theta}TM(qT)^{\ve},
  \]
  provided
  \begin{equation}\label{eqn: M Delta}
    T^{1/2-\theta}\ll M\ll T^{1/2}.
  \end{equation}
  As in the proof of Proposition \ref{prop: t}, we only need to
  bound \eqref{eqn: error1}, \eqref{eqn: error2E},
  and \eqref{eqn: MT<<} under our new assumptions
  \eqref{eqn: q Delta} and \eqref{eqn: M Delta}.
  It's easy to see that the bound in \eqref{eqn: error1} will be
  $q^{1/2-\theta}TM(qT)^\ve$ now.
  Moreover, \eqref{eqn: error2E} is easy to handle too.
  Now we consider \eqref{eqn: MT<<}. By \eqref{eqn: M Delta}, we have
  \[
    \tilde{\cR}_j^{\dag,2}
    \ll (qT)^\ve q^3 T^{j+\frac{3}{2}} M^{-3j-\frac{9}{2}}
    \ll (qT)^\ve q^3 T^{\frac{3}{2}} M^{-\frac{9}{2}}
    \ll (qT)^\ve q^{\frac{1}{2}-\theta} q^{\frac{5}{2}+\theta} T^{\frac{9}{2}\theta-\frac{3}{4}}.
  \]
  So we want $q\leq T^{(\frac{3}{4}-\frac{9}{2}\theta)/(\frac{5}{2}+\theta)}$,
  which coincides with \eqref{eqn: q Delta} with our choice of $\theta$.
  Now we complete the proof of Theorem \ref{thm: main}.
\end{proof}

\medskip
\noindent
{\bf Acknowledgements} \; The author would like to thank Professors Dorian Goldfeld, Jianya Liu, Ze\'ev Rudnick, and Wei Zhang for their valuable advice and constant encouragement.
He also wishes to thank Professor Matthew Young for explaining some
details in his paper~\cite{young2014weyl}.
He wants to thank the anonymous referees and editors
for their kind comments and valuable suggestions.
This work was finished when he was visiting Columbia University.
He is grateful to the China Scholarship Council (CSC) for
supporting his studies at Columbia University.
He also wants to thank 
the Department of Mathematics at Columbia University for its hospitality.
This work is partly supported by NSFC grant 11531008 and
IRT\_16R43 from the Ministry of Education, China.




\end{document}